\documentclass[12pt,reqno]{amsart}
\usepackage{amssymb,latexsym,amsmath,epsfig,mathrsfs}
\usepackage{enumitem, amsmath}
\usepackage{thmtools}
\usepackage{rotating}
\usepackage{graphicx}
\usepackage{amssymb}
\usepackage{lineno}

\usepackage{cite}
\usepackage{multicol}
\usepackage[top=3cm, bottom=3cm, left=2.5cm, right=2.5cm]{geometry}
\usepackage[usenames]{color}
\usepackage[colorlinks=true,
linkcolor=blue,
filecolor=blue,
citecolor=blue]{hyperref}
\usepackage{thm-restate}

\numberwithin{equation}{section}

\newtheorem{prop}{Proposition}[section]

\newtheorem{example}{Example}[section]
\newtheorem{theorem}{Theorem}[section]

\newtheorem{lemma}[theorem]{Lemma}

\newtheorem{conjecture}[theorem]{Conjecture}

\begin{document}
\title[On small Doubling in right-ordered groups and  Baumslag-Solitar Groups-II]{On small Doubling in right-ordered groups and  Baumslag-Solitar Groups - II \\[2ex] \small \text{\textit{In the memory of G.A. Freiman(1926-2024)}}} 
	
	
	
	\author[Mohan]{Mohan}
		\address{School of Engineering and Technology, Vivekananda Institute of Professional Studies-Technical Campus, Delhi, 110034, India}
		\email{mohan98math@gmail.com}
	\author[Neetu]{Neetu$^{\dagger}$}
	\address{Department of Mathematical and Computational Sciences, National Institute of Technology Surathkal, Karnataka, 575025, India}
	\email{chananianeetu@gmail.com}
    \thanks{$^{\dagger}$The corresponding author}
	
	\subjclass[2010]{11P70, 20F05, 20F99, 11B13, 05E15}
	
	
	
	\keywords{Ordered groups, right-ordered groups, Baumslag-Solitar Groups, product set, inverse problem, Freiman's 3k-4 conjecture}
	
	\begin{abstract}  
Recently, Mohan et al. [Results Math. \textbf{80} (2025), No. 4, 122] answered Freiman's $3k-4$ conjecture 
in right-ordered groups under certain restrictions. In this paper, we take a step further by investigating the structure of nonempty subsets $S$ of a right-ordered group satisfying the small doubling condition $|S^2| = 3|S|-3$. Moreover, we provide a 
complete characterization of all nonempty finite subsets $S$ of the Baumslag-Solitar group $\mathrm{BS}(1,q)$ (with $q \in \mathbb{Z}$ and $q \neq -1$) for which 
$|S^2| = 3|S|-3$ and the identity element is the minimum of $S$.

\end{abstract}
\maketitle

\section{Definitions and Notations}
 Let $\mathbb{N}$ be the set of all natural numbers and $\mathbb{Z}$ be the set of all integers. For integers $a$ and $b$, we define $[\![ a,b ]\!] = \{x\in \mathbb{Z}: a\leq x \leq b\}$.   The cardinality of the set $S$ is denoted by $|S|$. Throughout the article, $e$ stands for the identity element of the group $G$. Let $S$ be a nonempty subset of a group $G$.  Then the centralizer of $S$, denoted by $C_{G}(S)$, is defined as follows:
    $$C_{G}(S) := \{x\in G: xs=sx ~\text{for all}~ s\in S\}.$$
The subgroup generated by the set $S$ is denoted by $\langle S \rangle$. In particular, the subgroup generated by $x_{1}, x_{2}, \ldots, x_{n}$ is denoted by $\langle x_{1}, x_{2}, \ldots, x_{n} \rangle$.  
 A set $S$ is said to be abelian if $\langle S \rangle$ is an abelian subgroup of $G$. Define  $[x,y] := x^{-1}y^{-1}xy$ and $x^{y} := y^{-1}xy$, where $x,y \in G$.
\par
A group $G$ is said to be a \textit{right-ordered group} if it is a totally ordered set under a relation $\leq$, which satisfies the following condition:
	$$xz \leq yz \ \text{whenever} \ x \leq y \ \text{for all} \ x, y, z \in G.$$
	\noindent  Similarly, we define \textit{left-ordered group}. A group $G$ is said to be a left-ordered group if it is a totally ordered set under a relation $\leq$, which satisfies the following condition:
	$$zx \leq zy \ \text{whenever} \ x \leq y \ \text{for all} \ x, y, z \in G.$$
	A group $G$ is said to be an \textit{ordered group} if it is both left and right-ordered group under a common relation $\leq$. A group $G$ is said to be orderable if it admits an ordering $\leq$ such that $wxz \leq wyz$ whenever $x \leq y$ for all $w,x,y,z \in G$. For example, Baumslag-Solitar group $\mathrm{BS}(1,q) :=\langle a,b\mid ab=b^qa\rangle$, where $q\in \mathbb{Z}$, is an ordered group when $q\geq 0$, but a right-ordered group (not orderable) when $q<0$ (see \cite{Neetu_Mohan_Shankar}). If $G$ is an abelian right-ordered group, then $G$ is an ordered group. For more details on right-ordered groups and ordered groups, one may see \cite{Conrad_1959},  \cite{1999_POS}, and \cite{Kopytov_1994}.  If $S$ is a nonempty finite subset of a right-ordered group $G$ with a relation $\leq$, then the maximum and minimum elements of the set $S$, denoted by $\max(S)$ and $\min(S)$, are defined as follows:
	$$\max(S) = x, ~\text{if}~ x\in S~\text{and}~x\geq s~\text{for all}~s\in S, $$ and $$\min(S) = y,~\text{if}~y\in S~\text{and}~y\leq s~\text{for all}~s\in S. $$
	An element $x$ of a right-ordered group $G$ is said to be positive (negative) if $x > e$ $(x < e)$.
\section{Introduction}

Let $\alpha$ and $\beta$ be real numbers with $\alpha \geq 1$. A finite subset $S$ of a group $G$ is said to satisfy the small doubling property if $$|S^{2}| \leq \alpha |S| + \beta,$$ where $S^{2}=\{s_{1}s_{2}\colon s_{1},s_{2}\in S\}$. In the abelian case, where the group operation is written additively, the corresponding sumset is denoted by  $2S=\{s_{1} + s_{2} \colon s_{1}, s_{2}\in S \}$. The inverse problem of small doubling is to describe the structure of such sets $S$, which is one of the central questions in additive combinatorics. The classical Freiman inverse theorems provide a complete description of a finite subset $S$ of an abelian group under specific small doubling constraints (see \cite{Freiman_1959, Freiman_1960, Freiman_1973, Freiman_1999, Nathanson_1996}). The first major result is Freiman's $3k-4$ theorem as follows.
\begin{theorem}\textup{\cite{Freiman_1973}}
Let $S$ be a finite set of integers with at least three elements. If $|2S| \leq 3|S| - 4$, then $S$ is contained in an arithmetic progression of size $|2S| - |S| + 1 \leq 2|S| - 3$.\end{theorem}
\noindent Further, the following Freiman's  $3k-3$ theorem describes the extremal case.
\begin{theorem}\textup{\cite{Freiman_1973}}
Let $S$ be a finite set of integers with at least three elements. If $|2S| = 3|S| - 3$, then one of the following holds:
\begin{itemize}
    \item[\textup{(i)}] $S$ is a subset of an arithmetic progression of size at most $2|S| - 1$,
   \item[\textup{(ii)}] $S$ is the union of two arithmetic progressions with the same difference,
    \item[\textup{(iii)}] $|S| = 6$ and $S$ is Freiman-isomorphic to the set
    $$
    K_6 = \{(0,0), (1,0), (2,0), (0,1), (0,2), (1,1)\}.
    $$
\end{itemize}
\end{theorem}
\noindent In \cite{Freiman_1973}, Freiman also investigated subsets $S \subseteq \mathbb{Z}$ with $|2S| = 3|S| - 2$, providing a deeper understanding of small doubling in torsion-free abelian groups. Freiman et al.\cite{Freiman_et_al_2014, Freiman_2012, Freiman_Herzog_2014, Freiman_et_al_2015, Freiman_et_al_2016, Freiman_et_al_2017, Freiman_et_al_2018} initiated the systematic study of small doubling in ordered groups.
By extending Freiman's $3k-4$ theorem to the ordered group setting, they established the following results.
\begin{theorem}\textup{\cite[Theorem 1.3]{Freiman_et_al_2014}}\label{Freiman's result for ordered group-1}
    Let $G$ be an orderable group and $S$ be a finite subset of $G$ of size at least $3$. If $|S^{2}| \leq 3|S| - 3$, then $\langle S \rangle$ is an abelian subgroup of $G$.\end{theorem}
   
\begin{theorem}\textup{\cite[Corollary 1.4]{Freiman_et_al_2014}}\label{Freiman GP}
Let $G$ be an orderable group and $S$ be a finite subset of $G$ with at least three elements. If $|S^2| \leq 3|S| - 4$, then there exist $x, g \in G$, with $gx = xg$, such that
$$ S \subseteq \{x, xg, xg^2, \dots, xg^t\}, \quad \text{where } t = |S^2|-|S|.$$
\end{theorem}

\noindent A more detailed structural result for such sets is given in \cite{Freiman_et_al_2017}.
\begin{theorem}\textup{\cite[Theorem 4]{Freiman_et_al_2017}}\label{ordered 3k-3}
   Let $G$ be an ordered group and $S$ be a 
   finite subset of $G$. Then the
following statements hold.
\begin{itemize}
    \item[\upshape(i)] $|S^{2}|\geq 2|S|-1$.
    \item[\upshape(ii)] If $2|S|-1\leq |S^{2}|\leq 3|S|-4$, then $\langle S\rangle$ is abelian and at most $2$-generated.
     \item[\upshape(iii)] If $ |S^{2}|=3|S|-3$, then $\langle S\rangle$ is abelian and at most $3$-generated.
      \item[\upshape(iv)] Let $ |S^{2}|=3|S|-3+b$ for some integer $b\geq 1$. Then either $|S|=4$, $b=1$,  $\langle S\rangle$ is abelian and at most $(b+3)$-generated or $\langle S\rangle$ is at most $(b+2)$-generated .
\end{itemize}
\end{theorem}
 \noindent A complete structural classification of sets $S$ with $|S^2| = 3|S| - 2$ in ordered groups, including the nonabelian case, was subsequently obtained in \cite{Freiman_et_al_2016, Freiman_et_al_2017}. These works reveal that such sets generate either abelian subgroups of bounded rank or specific nonabelian groups such as nilpotent groups of class $2$ or the Baumslag-Solitar group $\mathrm{BS}(1,2)$.

  More recently, analogous structural results have been established for various groups, including nilpotent groups of class 2, ordered semigroups, groups with bounded torsion, groups of prime order, cyclic groups, $m$-Engel groups, and Baumslag–Solitar groups  $\mathrm{BS}(1,q)  =\langle a,b: ab = b^qa \rangle$, $q\geq 0$ (see, for instance, \cite{Abdollahi_Jafari_2020, Chahal_Kaur_2025, Freiman_et_al_2018, Lev_2022, Hui_2022, Bol_2023, Co_2023, Lev_2023, Lev_2023.1, Lev_2020, Longobardi_2020, Pandey_2017, Tringali_2015, Sanders_2013}). 
  Despite this substantial progress, comparatively little is known in the setting of right-ordered groups, especially the Baumslag-Solitar groups $\mathrm{BS}(1,q)  =\langle a,b: ab = b^qa\rangle$, where $q<0$ is an integer. To the best of our knowledge, the first systematic study in this direction was undertaken by the authors together with Shankar \cite{Neetu_Mohan_Shankar}, where several foundational results on small doubling in right-ordered groups and the Baumslag-Solitar groups $\mathrm{BS}(1,q)$ were established.
 
    \begin{theorem}\textup{{\cite[Theorem 3.1]{Neetu_Mohan_Shankar}}}
		 Let $A$ be a nonempty finite abelian subset of a right-ordered group $G$ with $|A|\geq3$. Let $y\in G\setminus A$ and $S=A\cup\{y\}$. If 
		\begin{equation}\label{Sec-3-Eq-1}
			|S^{2}|\leq 3|S|-4,
		\end{equation}
		then $\langle S\rangle$ is an abelian subgroup of $G$. The bound in \eqref{Sec-3-Eq-1} is the best possible. 
	\end{theorem}
    \begin{restatable}{theorem}{MainTheorem}\textup{{\cite[Theorem 3.3]{Neetu_Mohan_Shankar}}}\label{NMS-Thm-3.3}
 Let $A$ and $B$ be two nonempty abelian subsets of a right-ordered group $G$  such that  $\max(A) < \min(B)$ and $\max(A^{2})\leq \min(B^{2})$. If 
    \begin{equation*}
        \left|(A\cup B)^{2}\right| \leq 3|A\cup B|-4,
    \end{equation*}
    then  $\langle A\cup B\rangle$ is an abelian subgroup of $G$.
\end{restatable}

    \begin{theorem}\textup{\cite[Theorem 4.6]{Neetu_Mohan_Shankar}}\label{Theorem for BS(1,q)}
		Let $q$ be an integer with $q \neq -1$.  Let $S$ be a nonempty finite subset of $\mathrm{BS}(1,q)$ with $\left|S\right| \geq 3$ and $e =\min(S)$. If 
		\begin{equation*}
			\left|S^{2}\right| \leq 3|S|-4,
		\end{equation*}
		then $\langle S\rangle$ is abelian. \end{theorem}
	
The aim of this article is to go one step further, that is, to study the inverse problems when $|S^{2}|=3|S|-3$. Mainly, we present 
 a complete characterization of the structure of a set $S$ subset of a right-ordered group, satisfying $|S^{2}|=3|S|-3$ under certain restrictions on the set $S$. Additionally, we describe the structure of the finite set $S$, a subset of the Baumslag-Solitar group $\mathrm{BS}(1,q)=\langle a,b\mid ab=b^qa\rangle$, where $q\in \mathbb{Z}$ and $q\neq -1$ with the identity element as their minimum, satisfying $|S^2|=3|S| - 3$. In particular, we prove the following results.
 \begin{theorem}\label{Theorem-1}
   Let $A$ be a nonempty finite abelian subset of a right-ordered group $G$ with $|A|\geq3$. Let $S=A\cup \{y\}$, where $y\in G\setminus A$. If $|S^{2}|=3|S|-3$, then either $\langle S\rangle$ is abelian, or there exist elements $x,g\in G$ such that $S$ is of the form $$S=\{x,xg,\ldots,xg^{k-2}, y\},$$ where $k=|S|$ and one of the following holds:
    \begin{enumerate}
    \item[\upshape(i)]   $[xg^{i},y]=g^{k-2-2i}$ for $i\in  [\![ 0,k-2 ]\!]$;
       \item[\upshape(ii)] $[x,y]=1, [y,g^{2}]=1$, and  $|S|=4$.
    \end{enumerate} 
\end{theorem}
 \begin{theorem}\label{Theorem-2}
 Let $A$ and $B$ be two nonempty abelian subsets of a right-ordered group $G$ with  $\max(A) < \min(B)$, $\max(A^{2})\leq \min(B^{2})$, $|A|\geq 4$, and $|B|\geq 3$.  If 
    \begin{equation*}
        \left|(A\cup B)^{2}\right| = 3|A\cup B|-3,
    \end{equation*}
    then  $\langle A\cup B\rangle$ is an abelian subgroup of $G$.
\end{theorem}
 \noindent The proof of Theorem \ref{Theorem-1} and Theorem \ref{Theorem-2} is given in Section \ref{section-3} and Section \ref{section-4}, respectively.
 \begin{theorem}\label{Theorem-3}
  Let $q$ be an integer with $q \neq -1$.  Let $S$ be a nonempty finite subset of $\mathrm{BS}(1,q)$ with $\left|S\right|=k\geq 5$ and  $e =\min(S)$. If $\left|S^{2}\right| = 3k-3$, then  either $\langle S\rangle$ is an  abelian subgroup, or there exist elements $x,g\in \mathrm{BS}(1,q)$ such that $S$ is of the form $$S=\{e, x, xg, xg^{2}, xg^{3},\ldots, xg^{k-2}\},$$ with $xg\neq gx$ and $(xg^{i})^{2}=x^{2}$ for each $0\leq i\leq k-2$.
   
\end{theorem}
\noindent The proof of the above theorem is given in Section \ref{section-6}.

\section{Proof of Theorem \ref{Theorem-1}}\label{section-3}

To prove Theorem \ref{Theorem-1}, the following auxiliary propositions are required.
\begin{prop}\textup{\cite[Proposition 3.1]{Neetu_Mohan_Shankar}}\label{Prop-3.1}
		Let $G$ be a right-ordered group, and let $A$ be a finite subset of $G$ with $|A|\geq 2$. Then \begin{equation*}\label{2k-1}
			|A^{2}|\geq 2|A|-1.
		\end{equation*} Further, if  $|A^{2}|=2|A|-1$, then there exist $x,g\in G$ such that $g\neq e$ and $A=\{x, xg,\ldots,xg^{|A|-1}\}$ with either $xg=gx$ or $xgx^{-1} = g^{-1}$.
	\end{prop}
    \begin{prop}\label{2|A|}
         Let $A$ be a nonempty finite abelian subset of a right-ordered group $G$ such that $|A^2| = 2|A|$. Then $$A = \{x,xg,xg^{2}, \ldots, xg^{|A|}\}\setminus \{y\},$$ where $y\in \{xg,xg^{|A|-1}\}$.  
    \end{prop}
   \begin{proof}
      Since  $A$ is abelian, the subgroup generated by $A$ is an ordered group. Using Theorem \ref{Freiman GP}, we get that $A\subseteq \{x,xg, \ldots,xg^{|A|}\}$, where $x,g \in G$. It is easy to see that $|A^2| =2|A|$ if $A = \{x,xg,xg^{2}, \ldots, xg^{|A|}\}\setminus \{y\}$, where $y\in \{xg,xg^{|A|-1}\}$, otherwise $|A^2| = 2|A|+1$.
   \end{proof}
    
\begin{prop}\textup{\cite[Proposition 2.3]{Neetu_Mohan_Shankar}}\label{NMS-Pro-1} 
Let $A$ be a nonempty finite subset of a right-ordered group $G$.  Let  $y\in G$ such that $y>\max(A)$ and  $y^{2}\leq \max(A^{2})$. Then $$|yA\cup Ay|\geq |A|+1.$$      \end{prop}

\begin{prop}\textup{\cite[Proposition 2.4]{Neetu_Mohan_Shankar}}\label{NMS-Pro-2} 
Let $A$ be a nonempty finite subset of a right-ordered group $G$.  Let  $y\in G$ 
 such that $y<\min(A)$ and  $y^{2}\geq \min(A^{2})$. Then $$|yA\cup Ay|\geq |A| +1. $$ \end{prop} 
 \noindent It is easy to deduce the following proposition from the above two propositions.
\begin{prop}\label{Coro-1}
    Let $A$ be a nonempty finite subset of a right-ordered group $G$.  Let  $y\in G$ 
 such that $y^{2}\in A^{2}$, and  either $y<\min(A)$ or $y>\max(A)$. Then $|yA\cup Ay|\geq |A|+1$.    
\end{prop}
\begin{proof} Since $y^{2}\in A^{2}$, we have $$\min(A^{2}) \leq y^{2} \leq \max(A^{2}).$$ Therefore, by Proposition \ref{NMS-Pro-1} and Proposition \ref{NMS-Pro-2}, we have $|yA\cup Ay|\geq |A|+1$.
\end{proof}
\begin{prop}\label{prop-1} 
   Let $A=\{x,xg,xg^{2},\ldots,xg^{k-1}\}$ be a subset of $k\geq 3$ elements of a right-ordered group $G$ with $xg=gx$ and $g\neq e$.  Let $y\in G\setminus C_G(A)$ such that $yA= Ay$.  Then $y^{2}\notin A^{2}$. \end{prop}
   \begin{proof}   If $y>\max(A)$ or $y<\min(A)$, then by Proposition \ref{Coro-1}, we get that $y^{2}\notin A^{2}$. Let $\min(A)<y<\max(A)$. We also assume $g > e$; the case $g < e$ is similar. This implies $$x<xg<xg^{2}<\cdots<xg^{k-1}.$$  
    Therefore, the elements of the set $Ay  = \{xy,xgy,xgy^{2},\ldots,xg^{k-1}y\}$ can be arranged only in the following way:
	\begin{equation}\label{Seq-1}
		xy<xgy<xg^{2}y<\cdots<xg^{k-1}y.
	\end{equation} Now, if $yx<yxg$, then the elements of the set $yA=  \{yx,yxg,yxg^{2},\ldots,yxg^{k-1}\}$ can only be arranged in the following way:
	\begin{equation}\label{Seq-2}
		yx<yxg<yxg^{2}<\cdots<yxg^{k-1}.
	\end{equation}  Since $yA=Ay$, we obtain $xy=yx$ and $yxg=xgy$ by comparing \eqref{Seq-1} and \eqref{Seq-2}. This implies  $y\in C_{G}(A)$, which is a contradiction. Therefore, $yx>yxg$ and the elements of the set $yA$ can be arranged in the following way:
\begin{equation}\label{Seq-3}
	yx>yxg>yxg^{2}>\cdots>yxg^{k-1}.
\end{equation} Again, since $yA=Ay$, we obtain  \begin{equation}\label{thm-1.9-1.3}
		xg^{i}y=yxg^{k-1-i},
	\end{equation} where $i \in [\![0,k-1]\!]$, by comparing \eqref{Seq-1} and \eqref{Seq-3}.  This implies that \begin{equation}\label{thm-1.9-2.1.1}
		yg^{-1}=gy.
	\end{equation} 
 Now we show that $y^{2}\notin A^{2}$. Suppose $y^{2}\in A^{2}$. Then \begin{equation}\label{thm-1.9-3.1}
		y^{2}=x^{2}g^t \text{ for some } t\in [\![0,2k-4]\!].
	\end{equation} This, together with  \eqref{thm-1.9-2.1.1}, gives  $$yg^ty=x^{2}.$$  Using \eqref{thm-1.9-1.3}, we obtain $$yg^ty^{2}=x^{2}y=x(xy)=x(yxg^{k-1})=(xy)xg^{k-1}=yxg^{k-1}xg^{k-1}.$$ This gives that $y^{2}=x^{2}g^{2k-2-t}$, and by \eqref{thm-1.9-3.1}, we have $$x^{2}g^t=x^{2}g^{2k-2-t}.$$ This implies  $t=k-1$. Therefore, $y^{2} = x^{2}g^{k-1}<yxg^{k-1}=xy$, which is a contradiction because $xy<y^{2}$. Hence, $y^{2}\notin A^{2}$. 
\end{proof}
\begin{prop}\label{prop-3.8}
    Let $A$ be a nonempty finite abelian subset of a right-ordered group $G$ such that $|A^{2}|=2|A|$ with $|A|\geq 3$.  Let $y\in G\setminus C_G(A)$  such that $\min(A)<y$. Then  $|yA\cup Ay|\geq|A|+1$.  

\end{prop}
\begin{proof}Let $|A|=k$. Since $|A^{2}|=2|A|$ and $A$ is abelian, it follows from Proposition \ref{2|A|} that   $$A=\{x,xg,xg^{2},\ldots, xg^{k-1}, xg^{k}\} \setminus \{z\},$$ where $z\in \{xg, xg^{k-1}\}$ for some $x,g\in G$ with $xg=gx$ and $g\neq e$.  Let $g>e$; the case $g<e$ is similar. Then $$x<xg<xg^{2}<\cdots<xg^{k-1}<xg^{k}.$$ To get a contradiction, assume that  $|yA\cup Ay|=|A|$. This implies $yA = Ay$. Now, consider the following cases. \par

\textbf{Case 1} $(z=xg)$. In this case, we have  $ Ay=\{xy,xg^{2}y, xg^{3}y,\ldots,  xg^{k-1}y,  xg^{k}y \},$ with \begin{equation}\label{prop-3.7-eq-2}
xy<xg^{2}y<xg^{3}y<\cdots<xg^{k-1}y<xg^{k}y.    
\end{equation} 
Now, if $yx<yxg$, then the elements of $yA = \{yx, yxg^{2}, yxg^{3},\ldots,yxg^{k-1}, yxg^{k}\}$ can be arranged in the following way: \begin{equation}\label{prop-3.7-eq-5}
    yx<yxg^{2}<yxg^{3}<\cdots<yxg^{k-1}< yxg^{k}.
\end{equation} 
 Since $yA= Ay$, we have $xy=yx$, $xg^{2}y=yxg^{2}$ and $xg^{3}y=yxg^{3}$ by  comparing \eqref{prop-3.7-eq-2} and \eqref{prop-3.7-eq-5}. This implies $y\in C_{G}(A)$, which is not possible. Therefore, $yx>yxg$ and the elements of $yA$ can be arranged in the following way:
 \begin{equation}\label{prop-3.7-eq-6}
yx>yxg^{2}>yxg^{3}>\cdots>yxg^{k-1}> yxg^{k}.
\end{equation}
Again, since $yA=Ay$, we have  	$$xg^{i}y=yxg^{k+1-i},$$ where $i \in [\![2,k-1]\!]$, $xy=yxg^{k}$ and $yx=xg^{k}y$ by comparing \eqref{prop-3.7-eq-2} and \eqref{prop-3.7-eq-6}. From $xy=yxg^{k}$
    and $xg^{2}y=yxg^{k-1}$, we obtain \begin{equation}\label{thm-1.9-2}
		g^{2}yg=y.
	\end{equation} From $xg^{2}y=yxg^{k-1}$ and $xg^{3}y=yxg^{k-2}$, we obtain  \begin{equation}\label{thm-1.9-3}
       gyg=y.
	\end{equation} By combining \eqref{thm-1.9-2} and \eqref{thm-1.9-3}, we get that $g=e$, which is not possible.\\

\textbf{Case 2} $(z=xg^{k-1})$. In this case, we have  $$ Ay=\{xy,xgy,xg^{2}y,\ldots,xg^{k-2}y, xg^{k}y \},$$ with \begin{equation}\label{prop-3.7-eq-1}
    xy<xgy<xg^{2}y<\cdots<xg^{k-2}y<xg^{k}y. 
\end{equation}
Now, if $yx<yxg$, then the elements of  $yA = \{yx, yxg,  yxg^{2},\ldots,yxg^{k-2}, yxg^{k}\}$  can be arranged in the following way: \begin{equation}\label{prop-3.7-eq-3}
    yx<yxg<yxg^{2}<\cdots<yxg^{k-2}< yxg^{k}.
\end{equation} 
 Since $yA= Ay$, we obtain $xy=yx$ and $xgy=yxg$ by comparing \eqref{prop-3.7-eq-1} and \eqref{prop-3.7-eq-3}. This implies $y\in C_{G}(A)$, which is not possible. Therefore, $yx>yxg$ and the elements of  $yA$  can be arranged in the following way: \begin{equation}\label{prop-3.7-eq-4}
 yx>yxg>yxg^{2}>\cdots>yxg^{k-2}> yxg^{k}.
\end{equation}Again, since $yA=Ay$, we obtain \begin{equation*}
		xg^{i}y=yxg^{k-1-i},
	\end{equation*} where $i \in [\![1,k-2]\!]$, $xy=yxg^{k}$ and $yx=xg^{k}y$ by comparing \eqref{prop-3.7-eq-1} and \eqref{prop-3.7-eq-4}. From $xy=yxg^{k}$
    and $xgy=yxg^{k-2}$, we obtain \begin{equation}\label{thm-1.9-2.1}
		gyg^{2}=y.
	\end{equation} From  $xgy=yxg^{k-2}$ and $xg^{2}y=yxg^{k-3}$, we obtain  \begin{equation}\label{thm-1.9-3.2}
		gyg=y.
	\end{equation} By combining \eqref{thm-1.9-2.1} and \eqref{thm-1.9-3.2}, we get that $g=e$, which is again not possible. Hence, $|yA\cup Ay|\geq |A|+1$.
\end{proof}
\begin{prop}\label{prop-2}
   Let $A$ be a nonempty finite abelian subset of a right-ordered group $G$ with $|A|\geq3$. Let $S=A\cup \{y\}$, where $y\in G\setminus A$ such that $y>\max(A)$ or $y<\min(A)$. If 
   \begin{equation}\label{Prop-2-eq-1}
       |S^{2}|=3k-3;
   \end{equation} where $k=|S|$, then one of the following holds:
   \begin{enumerate}
       \item [\upshape(i)]$\langle S\rangle$ is  an abelian subgroup of $G$;

       \item [\upshape(ii)]$S=\{x,xg,\ldots,xg^{k-2}, y\},$ with $[xg^{i},y]=g^{k-2-2i}$ for $i\in  [\![ 0,k-2 ]\!]$, where $x,g\in G$.  
   \end{enumerate}
\end{prop}
\begin{proof}Clearly,  \begin{equation}\label{Pro-1-Eq-1}
    S^{2}=A^{2}\cup yA\cup Ay \cup \{y^{2}\}.
    \end{equation}  If $y\in C_G(A)$, then $\langle S\rangle$ is abelian. Let $y\notin C_G(A)$. Then
    \begin{equation*}\label{Pro-1-Eq-2.1}
        (yA\cup Ay)\cap A^{2}=\emptyset.
    \end{equation*} 
Also, $y^{2}\notin yA \cup Ay$. Now,  if $y^{2}\in A^{2}$, then  by Proposition \ref{Coro-1} and Proposition \ref{Prop-3.1}, we have  
$$2|A|-1 \leq |A^{2}|\leq |S^{2}|-|yA\cup Ay|\leq 3k-3-k=2k-3=2|A|-1.$$  If $y^{2}\notin A^{2}$, then by Proposition \ref{Prop-3.1}, we have $$2|A|-1 \leq |A^{2}|\leq |S^{2}|-|yA\cup Ay \cup \{y^{2}\}| \leq 3k-3-k=2k-3=2|A|-1.$$ Thus, \begin{equation}\label{Pro-1-Eq-2.2}
    |A^{2}|=2|A|-1.
\end{equation} Therefore, by Proposition \ref{Prop-3.1}, we have \begin{equation*}\label{Pro-1-Eq-2.2.1}
    A=\{x,xg,xg^{2},\ldots,xg^{k-2}\},
\end{equation*} where $xg=gx$ and $g\neq e$. Let $g>e$; the case $g<e$ is similar. Then,  $$x<xg<xg^{2}<\ldots<xg^{k-2}.$$ 
    Further, from \eqref{Prop-2-eq-1}, \eqref{Pro-1-Eq-1}, and \eqref{Pro-1-Eq-2.2}, we have \begin{equation}\label{Pro-1-Eq-4.1}
        k-1 \leq |Ay\cup yA|\leq k.
    \end{equation}
    \textbf{Claim:} $y^{2}\notin A^{2}$.\par
   Suppose that $y^{2}\in A^{2}$. Then $\min(A^{2}) \leq y^{2} = x^{2}g^{t} \leq \max(A^{2})$ for some $t \in [\![0,2k-4]\!]$.  Since $xg=gx$, we have $y^{2}\in C_{G}(A)$. Now consider the following cases.
    
    \textbf{Case 1} ($y>\max(A)$).  In this case, \begin{equation}\label{prop-case-3-eq-1} xy < xgy < xg^{2}y <\cdots < xg^{k-2}y<y^{2}\leq \max(A^{2})=x^{2}g^{2k-4}<yxg^{k-2}. \end{equation} This, together with \eqref{Pro-1-Eq-4.1}, implies that $|Ay\cup yA| = k$ and \begin{equation*}\label{prop-case-3-eq-1.1.1}
        yA \cup Ay=\{xy,xgy,xg^{2}y,\ldots, xg^{k-2}y, yxg^{k-2}\}.
    \end{equation*}
 If $yx>yxg$, then   $$yx>yxg>\cdots>yxg^{k-2}>xg^{k-2}y>\cdots>xy.$$ This implies that  $|yA\cup Ay|>k$, which contradicts \eqref{Pro-1-Eq-4.1}. Therefore, we have \begin{equation}\label{prop-case-3-eq-2}
     yx<yxg < yxg^{2} < yxg^3 <\cdots < yxg^{k-2}.
 \end{equation} Further,  if $yx\geq xg^{2}y$, then $$xy<xgy<xg^{2}y \leq yx<yxg<\cdots<yxg^{k-2},$$ which again yields $|yA\cup Ay|> k$. Therefore $yx<xg^{2}y$, and similarly $yxg<xg^{3}y$.  Since $yx, yxg \in Ay \cup yA$, it follows from \eqref{prop-case-3-eq-1} that  $$yx\in \{xy, xgy\}  \text{ and } yxg \in \{xy,xgy,xg^{2}y\}.$$ Consider the following possible cases.
 \begin{enumerate}
     \item If $yx=xy$ and $yxg=xgy$, then  $y\in C_G(A)$, which is a contradiction.
     \item If  $yx=xy$ and $yxg=xg^{2}y$, then $$yxgy=xg^{2}y^{2}=y^{2}xg^{2},$$ because $y^{2} \in C_{G}(A)$. This gives $xgy = yxg^{2}$. Therefore, from \eqref{prop-case-3-eq-2}, we have $$xg^{2}y = yxg<yxg^{2}=xgy,$$ which contradicts \eqref{prop-case-3-eq-1}.
     \item If $yx=xgy$ and  $yxg=xg^{2}y$, then we have $yxy=xgy^{2}=y^{2}xg$ because $y^{2} \in C_{G}(A)$. This implies  $xy=yxg$. Consequently, $$xgy = yx <yxg=xy,$$ which contradicts \eqref{prop-case-3-eq-1}.
 \end{enumerate}

 \textbf{Case 2}($y<\min(A)$). Since $y^{2}\in A^{2}$, we have  
 \begin{equation}\label{prop-case 4-eq-1}
     yx<x^{2}=\min(A^{2})\leq y^{2}<xy< xgy < xg^{2}y <\cdots < xg^{k-2}y.
 \end{equation} Therefore, $$yA \cup Ay=\{yx,xy,xgy,xg^{2}y,\ldots, xg^{k-2}y\}.$$  
 If $yx>yxg$, then $$xg^{k-2}y>\cdots>xgy>xy>yx>yxg>yxg^{2}>\cdots> yxg^{k-2}.$$ This implies that $|yA\cup Ay|>k$, which is a contradiction to \eqref{Pro-1-Eq-4.1}. Therefore, $yx<yxg$.
 Using a similar argument, we obtain that $yxg \leq xgy$ and   $yxg^{2} \leq xg^{2}y$. Therefore  $yxg \in  \{xy, xgy\}$ and $yxg^{2} \in \{xy, xgy, xg^{2}y\}$. Consider the following possible cases. \begin{enumerate}
     \item  If $yxg=xy$, then $yxgy=xy^{2}$. Since $y^{2}\in A^{2}$, we have $yxgy=xy^{2}=y^{2}x$, which implies $xgy = yx$.  This gives $yx<yxg=xy<xgy=yx$, which is not possible.
     \item If $yxg = xgy$ and  $yxg^{2}= xg^{2}y$, then $yxg^{2} = xgyg$. This implies $xy=yx$, which is a contradiction to \eqref{prop-case 4-eq-1}.
 \end{enumerate}
  
Hence, $y^{2}\notin A^{2}$.  Consequently, we have $|yA \cup Ay|=k-1$ and $yA=Ay$.  The elements of $Ay= \{xy,xgy,xg^{2}y,\ldots,xg^{k-2}y\}$ can be arranged in the following way:
	 \begin{equation}\label{Pro-1-Eq-2}
  xy<xgy<xg^{2}y<\cdots<xg^{k-2}y.
    \end{equation} Now, if $yx<yxg$, then the elements of $yA=\{yx,yxg,yxg^{2},\ldots,yxg^{k-2}\}$ can be arranged in the following way:
	 \begin{equation}\label{Pro-1-Eq-3}
        yx<yxg<yxg^{2}<\cdots<yxg^{k-2}.
        \end{equation} 
Since $yA=Ay$, we obtain $xy=yx$ and $yxg=xgy$ by comparing \eqref{Pro-1-Eq-2} and \eqref{Pro-1-Eq-3}. This implies  $y\in C_{G}(A)$, which is a contradiction. Therefore, $yx>yxg$ and the elements of the set $yA$ can be arranged in the following way:
\begin{equation}\label{Pro-1-Eq-4}
yx>yxg>yxg^{2}>\cdots>yxg^{k-2}.
 \end{equation} 
Since $yA=Ay$, we obtain $xg^{i}y=yxg^{k-i-2}$, where  $i\in  [\![ 0,k-2 ]\!]$ by comparing \eqref{Pro-1-Eq-2} and \eqref{Pro-1-Eq-4}. Using $xy=yxg^{k-2}$ and $xgy=yxg^{k-3}$, we get that $yg^{-1}=gy$. Therefore,  $[xg^{i},y]=g^{k-2-2i}$ for $i\in  [\![ 0,k-2 ]\!]$.
\end{proof}
	

\begin{prop}\label{prop-3}
Let $A$ be a nonempty finite abelian subset of a right-ordered group $G$ with $|A|\geq3$. Let $S=A\cup \{y\}$, where $y\in G\setminus{A}$ such that $\min(A)<y<\max(A)$. If $|S^{2}|=3|S|-3$, then one of the following holds: 
\begin{enumerate}
    \item[\upshape(i)] $\langle S\rangle$ is an abelian subgroup of $G$;
    \item[\upshape(ii)] $k=4$ and there exists $x,g\in G$ such that  $S=\{x,xg,xg^{2}, y\}$ with $[x,y]=1$ and  $[y,g^{2}]=1$.
\end{enumerate}
\end{prop}
\begin{proof}
Let $|S|=k$. If $y\in C_G(A)$, then $\langle S\rangle$ is an abelian subgroup. Let  $y\notin C_G(A)$. Then 
\begin{equation}\label{Prop-3.9-Eq-1}
    (yA\cup Ay)\cap A^{2}=\emptyset \ \text{and} \ y^{2} \notin yA \cup Ay.
\end{equation} 
Clearly, 
\begin{equation}\label{Prop-3.9-Eq-2}
    S^{2}=A^{2}\cup(yA\cup Ay)\cup \{y^{2}\}.
\end{equation} 
This, together with \eqref{Prop-3.9-Eq-1} and Proposition \ref{Prop-3.1}, gives $$2|A|-1 \leq |A^{2}|\leq |S^{2}|-|Ay\cup yA|=3k-3-(k-1)=2k-2=2|A|.$$  By Proposition  \ref{prop-3.8},  if $|A^{2}|=2|A|$, then  $|Ay \cup yA| \geq |A|+1$, and hence $$|S^{2}|\geq |A^{2}|+|Ay \cup yA| \geq  2(k-1)+k-1+1=3k-2,$$ which exceeds the assumed value of $|S^
{2}|$. Therefore, \begin{equation}\label{Prop-3.9-Eq-3}
    |A^{2}|=  2|A|-1.
\end{equation} Thus, by Proposition \ref{Prop-3.1}, we have $$A=\{x,xg,xg^{2},\ldots,xg^{k-2}\},$$ where $xg=gx$ and $g\neq e$. We assume $g>e$; the case  $g<e$ is similar. Therefore, we have
\begin{equation}\label{Prop-3.9-Eq-4}
xy<xgy<xg^{2}y<\cdots<xg^{k-2}y.
\end{equation} Since $x<y<xg^{k-2}$, we have $$yx<xg^{k-2}x=x^{2}g^{k-2}<yxg^{k-2}.$$
If $yx>yxg$, then  $yx<yxg^{k-2}<yxg^{k-3}<\cdots<yxg<yx$, which is not possible. Therefore,  $yx<yxg$ and  \begin{equation}\label{Prop-3.9-Eq-5}
    yx<yxg<\cdots<yxg^{k-2}.
\end{equation}  If  $|yA\cup Ay|=k-1$, then  $Ay=yA$, and it follows from \eqref{Prop-3.9-Eq-4} and \eqref{Prop-3.9-Eq-5} that  $xy=yx$ and $yxg=xgy$. This implies  $y\in C_{G}(A)$, which is a contradiction.  Therefore $|yA\cup Ay| \geq k$. Further, from \eqref{Prop-3.9-Eq-1}, \eqref{Prop-3.9-Eq-2}, and \eqref{Prop-3.9-Eq-3}, we get that $$ |Ay\cup yA| = k \ \text{and} \ y^{2} \in A^{2}.$$ Observe that if $yx<xy$, then using arguments similar to those in case 2 of Proposition \ref{prop-2} results in a contradiction. Therefore $xy\leq yx$. Now, we show that $xy=yx$. Suppose $xy<yx$. Since $yx<yxg$, we have
 $$yA\cup Ay=\{xy,yx,\ldots, yxg^{k-2}\},$$ with  \begin{equation}\label{Proposition-1-Eq-5}
      xy<yx<yxg<\cdots<yxg^{k-2}.
  \end{equation}  Moreover, the elements of $Ay$ can only be arranged as \begin{equation}\label{Proposition-1-Eq-6}
      xy<xgy<xg^{2}y<\cdots<xg^{k-2}y.
  \end{equation} Since $yA \subseteq yA \cup Ay$, on comparing the elements of \eqref{Proposition-1-Eq-5} and \eqref{Proposition-1-Eq-6}, we get that $xgy \leq yxg$ and  $xg^{2}y \leq yxg^{2}$. Therefore, either $xgy=yx$ or $xgy=yxg$. Consider the following cases.
  \begin{enumerate}
  
      \item 
      If $xgy = yxg$, then $xg^{2}y=yxg^{2}$. This implies $xy=yx$, which is a contradiction. 
      \item If  $xgy = yx$, then $yxgy=y^{2}x$. Since $y^{2}\in A^{2}$, we have $yxgy=xy^{2}$. This gives $yxg=xy<yx$, which contradicts \eqref{Proposition-1-Eq-5}.
  \end{enumerate}
  
 Therefore, $xy=yx$.  Clearly, $xgy\neq yxg$, otherwise $y\in C_G(A)$. Now, consider the following cases.

 \textbf{Case 1} $(yxg<xgy)$. In this case, we have  $$ yA\cup Ay= \{yx, yxg, xgy, \ldots, xg^{k-2}y\},$$ with 
   \begin{equation}\label{Proposition-1-Eq-7}
 xy=yx<yxg<xgy<xg^{2}y<\cdots<xg^{k-2}y.      
   \end{equation}
Also, the elements of $yA$ can be arranged in the following way:
\begin{equation}\label{Proposition-1-Eq-8}
     yx<yxg<yxg^{2}<\cdots<yxg^{k-2}.
 \end{equation}
Since $yA \subseteq yA \cup Ay$, on comparing the elements of \eqref{Proposition-1-Eq-7} and \eqref{Proposition-1-Eq-8}, we get that $yxg^{2}\in \{xgy,xg^{2}y\}$. If $yxg^{2}=xgy$, then $yxg^{2}y=xgy^{2}$. Since $y^{2}\in A^{2}$, we have  $yxg^{2}y=y^{2}xg$. This gives  $xg^{2}y=yxg$, which contradicts $yxg<xgy$. Therefore, $yxg^{2}=xg^{2}y$, and this gives $yg^{2}=g^{2}y$. That is, $[y,g^{2}]=1$. Now, we show that $k=4$. Suppose $k>4$. Then $yxg^3\in yA\cup Ay$ and $yxg^3=xg^3y$. This gives $yg^3=xg^3$, together with $yg^{2}=xg^{2}$, we obtain that $yg=gy$.  This implies $y\in C_G(A)$, which is a contradiction. Therefore $k=4$.\\

\textbf{Case 2} $(xgy<yxg)$.  Then 
   $$ yA\cup Ay= \{xy, xgy,yxg,yxg^{2}, \ldots,yxg^{k-2}\},$$ with \begin{equation}\label{xgy<yxg}
       yx=xy<xgy<yxg<yxg^{2}<\cdots<yxg^{k-2}.
   \end{equation}
   The elements of $Ay$ can be arranged in the following way:
\begin{equation}\label{xgy<yxg-1}
xy<xgy<xg^{2}y<xg^{3}y<\cdots<xg^{k-2}y.
 \end{equation}
Since $Ay \subseteq yA \cup Ay$, on comparing the elements of \eqref{xgy<yxg} and \eqref{xgy<yxg-1}, we get that $xg^{2}y \in \{yxg, yxg^{2}\}$. If $xg^{2}y=yxg$, then $yxg^{2}y=y^{2}xg$. Since $y^{2}\in A^{2}$, we have $yxg^{2}y=xgy^{2}$. This gives  $yxg^{2}=xgy$, which is not possible because $xgy<yxg<yxg^{2}$. Therefore $xg^{2}y=yxg^{2}$, this gives $g^{2}y=yg^{2}$. That is, $[y,g^{2}]=1$. Now, we show that $k=4$. Suppose $k>4$. Then $yxg^3\in yA\cup Ay$ and $yxg^3=xg^3y$. This gives $yg^3=xg^3$, together with $yg^{2}=xg^{2}$, we get that $yg=gy$.  This implies $y\in C_G(A)$, which is a contradiction. Thus, $k=4$. This completes the proof.
\end{proof}
\begin{proof}[Proof of Theorem \ref{Theorem-1}]
   It directly follows from Proposition \ref{prop-2} and Proposition \ref{prop-3}.
\end{proof}
\section{proof of theorem \ref{Theorem-2}}\label{section-4}
To prove Theorem \ref{Theorem-2}, the following auxiliary lemmas and propositions are required.
 \begin{lemma}\label{Lemma 4.1}
    Let $A=\{x,xg,xg^{2},\ldots,xg^{k-1}\}$ be a subset of $k\geq 2$ elements of a group $G$.  Let $y\in G$ be such that $y\in C_{G}(\{xg^{i}, xg^{i+1}\})$ for some $i\in [\![0,k-2]\!]$. Then $y\in C_G(A)$.
\end{lemma}
\begin{proof}
Since  $y\in C_{G}(\{xg^{i}, xg^{i+1}\})$, we have $$xg^{i+1}y = yxg^{i+1} =xg^{i}yg.$$ This gives $gy=yg$. Therefore, $$yxg^{i} = xg^{i}y = xyg^{i}.$$ This implies $yx=xy$. 
\end{proof}
    \begin{prop}\label{section-4-prop-4.1}
         Let $A$ and  $B$ be two nonempty finite abelian subsets of a right-ordered group $G$ such that the following holds:
     \begin{enumerate}
         \item[\upshape(1)]   $\max(A) < \min(B)$,
         \item[\upshape(2)]$\max(A^{2})= \min(B^{2})$,
         \item[\upshape(3)] $\min(B) \notin C_{G}(A)$, 
         \item [\upshape(4)] $|A|\geq 2$, $|B|\geq 2$, and $|A\cup B| \geq 5$.
     \end{enumerate} Then $\left|(A\cup B)^{2}\right| \geq  3|A\cup B|-2$.
\end{prop}
\begin{proof} To get a contradiction, we assume $\left|(A\cup B)^{2}\right| \leq  3|A\cup B|-3$.
  Let $A=\{x_{1}, \ldots, x_{i}\}$ and $B = \{x_{i+1},\ldots,x_{k}\} $ such that $\langle A \rangle$ and $\langle B \rangle$ are  abelian subgroups of $G$ with
  \begin{equation}\label{Prop-4.1 Eq-01}
      x_{1}  < \cdots <x_{i}<x_{i+1} < \cdots < x_{k}.
  \end{equation} Since $x_{i+1}\notin C_G(A)$, we have $x_{i+1}\notin \langle A\rangle.$ 
   It follows that \begin{equation}\label{Sec-3-Eq-2..1}
      A^{2}\cap (x_{i+1}A\cup Ax_{i+1})=\emptyset.
  \end{equation}
Now, suppose that  $x_{i+1}x_{a}\in B^2$ for some $x_{a} \in A$. Then there exist  $x_{b_1},x_{b_2}\in B$ such that $$x_{i+1}x_{a}=x_{b_1}x_{b_2}.$$ This implies that $x_{a}=x_{i+1}^{-1}x_{b_1}x_{b_2}=x_{b_1}x_{i+1}^{-1}x_{b_2} \geq x_{b_2}$, which contradicts \eqref{Prop-4.1 Eq-01}. Therefore, \begin{equation}\label{Prop-4.1-Eq-01A}x_{i+1}A\cap B^2 = \emptyset.
\end{equation}  
Suppose that $x_{i+1}x_{a} = x_{b}x_{i}$ for some $x_{a}\in A$ and $x_{b}\in B$. This implies that $$x_{i}\geq x_{a} = x_{i+1}^{-1}x_{b}x_{i} = x_{b}x_{i+1}^{-1}x_{i} \geq x_{i}.$$ This gives $x_{a}=x_{i}$. Thus, 
\begin{equation}\label{Prop-4.1-Eq-0001}
    x_{i+1}A \cap Bx_{i} = \{x_{i+1}x_{i}\}.
\end{equation} Since $\max(A^{2}) = \min(B^{2})$ and $A$ and $B$ are abelian sets, we have  $A^{2}\cap B^{2}=\{x_{i}^2\}$. Therefore, it follows from Proposition \ref{Prop-3.1} that 
\begin{equation}\label{Sec-3-Eq-12}
			|A^{2}\cup B^{2}|\geq |A^{2}| + |B^{2} |-1\geq 2i-1+2(k-i)-1-1=2k-3.
\end{equation}
		Further, note that  \begin{equation}\label{Sec-3-Eq-12.1}
		    \max (Ax_{i+1})=x_{i}x_{i+1}
		<x_{i+1}^2 = \max(A^2)=\min(B^2)=x_{i}^2<x_{i+1}x_{i} = \min(Bx_{i}).\end{equation} This implies \begin{equation}\label{Sec-3-Eq-13}
			(A^{2}\cup Ax_{i+1})\cap Bx_{i}=\emptyset
		\end{equation} 
		and 
		\begin{equation}\label{Sec-3-Eq-14}
			(B^{2}\cap Ax_{i+1})=\emptyset.
		\end{equation}
        From \eqref{Prop-4.1-Eq-0001} and \eqref{Sec-3-Eq-13}, we have $Ax_{i+1} \neq x_{i+1}A$, which implies \begin{equation}\label{Prop-4.1-Eq-000101}
            |Ax_{i+1} \cup x_{i+1}A| \geq |A| +1.
        \end{equation}
		Also, from \eqref{Sec-3-Eq-12.1}, we have $x_{i}x_{i+1} \neq x_{i+1}x_{i}$ and $x_{i}\notin \langle B \rangle$. This gives   \begin{equation}\label{Sec-3-Eq-15}
			B^{2} \cap Bx_{i} = \emptyset.
		\end{equation}
		By combining  \eqref{Sec-3-Eq-2..1}, \eqref{Sec-3-Eq-12}, \eqref{Sec-3-Eq-13}, \eqref{Sec-3-Eq-14}, and \eqref{Sec-3-Eq-15}, we get that 
		\begin{align*}
			3k-3\geq|(A\cup B)^2| & \geq |A^{2}\cup B^{2}\cup Ax_{i+1}\cup Bx_{i}|\\ & =|A^{2}\cup B^{2}|+|Ax_{i+1}|+|Bx_{i}| \\ &\geq 2k-3+i+k-i=3k-3.
		\end{align*} Thus, $$(A\cup B)^{2}=A^{2}\cup B^{2}\cup Ax_{i+1}\cup Bx_{i}$$ and $$|A^{2}|=2i-1.$$ Therefore, by Proposition \ref{Prop-3.1}, we have $$A=\{x,xg,xg^2,\ldots, xg^{i-1}\},$$ with $xg=gx$ and $g\neq e$ for some $x,g\in G$. Without loss of generality, we assume $g>e$. This implies  $$x_{j} = xg^{j-1} \ \text{for} \ j\in [\![0,i]\!].$$ Further, \eqref{Sec-3-Eq-2..1}, \eqref{Prop-4.1-Eq-01A}, \eqref{Prop-4.1-Eq-0001}, \eqref{Sec-3-Eq-13}, \eqref{Sec-3-Eq-14}, and \eqref{Sec-3-Eq-15} imply that   $$|A^{2}\cup B^{2}\cup Ax_{i+1}\cup Bx_{i}|=|(A \cup B)^{2}|\geq |A^{2}\cup B^{2}|+|Ax_{i+1} \cup x_{i+1}A|+|Bx_{i}|-1.$$ This gives  
        \begin{equation*}
            |Ax_{i+1}\cup x_{i+1}A|\leq |A|+1,
        \end{equation*}
       and  it follows from \eqref{Prop-4.1-Eq-000101} that
        \begin{equation}\label{Prop 4.1A Eq-1}
            |Ax_{i+1}\cup x_{i+1}A|=|A|+1.
        \end{equation} Therefore, by \eqref{Sec-3-Eq-12.1}, we have \begin{equation}\label{Prop 4.1A Eq-1.1}
            Ax_{i+1}\cup x_{i+1}A= Ax_{i+1}\cup \{x_{i+1}x_{i}\}.
        \end{equation}
        Clearly, $x_{i+1}x_{2}\neq x_{i+1}x_{1}$. If $x_{i+1}x_{2}<x_{i+1}x_{1}$, then we have  $$x_{1}x_{i+1}<x_{2}x_{i+1}<\cdots<x_{i}x_{i+1}<x_{i+1}^{2}=x_{i}^{2}<x_{i+1}x_{i}<x_{i+1}x_{i-1}<\cdots<x_{i+1}x_{2}<x_{i+1}x_{1}.$$
    This implies that  $|Ax_{i+1}\cup x_{i+1}A| > |A|+1,$ which contradicts \eqref{Prop 4.1A Eq-1}. Therefore, $x_{i+1}x_{1}<x_{i+1}x_{2}$, which implies
    \begin{equation*}\label{Prop-4.1 Eq-11}
       x_{i+1}x_{1}<x_{i+1}x_{2}<\cdots<x_{i+1}x_{i}.
    \end{equation*} 
   Further, if $x_{i+1}x_{1}\geq x_{3}x_{i+1}$, then $$x_{1}x_{i+1}<x_{2}x_{i+1}<x_{i+1}x_{1}<x_{i+1}x_{2}<\cdots<x_{i+1}x_{i},$$ and this gives $|Ax_{i+1}\cup x_{i+1}A| > |A|+1,$ which contradicts \eqref{Prop 4.1A Eq-1}. Therefore, 
   \begin{equation}\label{Prop-4.1 Eq-12}
       x_{i+1}x_{1}< x_{3}x_{i+1}.
   \end{equation}
   Similarly, if $x_{i+1}x_{2}>x_{3}x_{i+1}$, then we have $$x_{1}x_{i+1}<x_{2}x_{i+1}<x_{3}x_{i+1}<x_{i+1}x_{2}<\cdots<x_{i+1}x_{i},$$ which forces $|Ax_{i+1}\cup x_{i+1}A| > |A|+1$, which contradicts \eqref{Prop 4.1A Eq-1}. Therefore, $x_{i+1}x_{2}\leq x_{3}x_{i+1}$. Further, if $x_{i+1}x_{2}=x_{3}x_{i+1}$, then $$ x_{i+1}x_{2}x_{i+1}=x_{3}x_{i+1}^{2} =   x_{3}x_{i}^2 = x_{i}^{2}x_{3} = x_{i+1}^{2}x_{3}.$$ This gives $x_{2}x_{i+1}=x_{i+1}x_{3}$, and  therefore,  $$x_{3}x_{i+1}=x_{i+1}x_{2}<x_{i+1}x_{3}=x_{2}x_{i+1}.$$ This implies $x_{3}<x_{2}$, which is a contradiction. Therefore, \begin{equation}\label{Prop-4.1 Eq-13}
       x_{i+1}x_{2}<x_{3}x_{i+1}.
   \end{equation} Since $x_{i+1}x_{1}, x_{i+1}x_{2} \in Ax_{i+1} \cup x_{i+1}A = Ax_{i+1} \cup 
   \{x_{i}x_{i+1}\}$, it follows from \eqref{Prop 4.1A Eq-1}, \eqref{Prop 4.1A Eq-1.1}, \eqref{Prop-4.1 Eq-12}, and \eqref{Prop-4.1 Eq-13} that 
   \begin{equation*}
      \{x_{i+1}x_{1}, x_{i+1}x_{2}\} = \{x_{1}x_{i+1}, x_{2}x_{i+1}\}.
   \end{equation*} This implies that
   $x_{i+1}x_{1}=x_{1}x_{i+1}$ and $x_{i+1}x_{2}=x_{2}x_{i+1}$. Therefore, by Lemma  \ref{Lemma 4.1}, we obtain that $x_{i+1}\in C_G(A)$, which is a contradiction.
\end{proof}

    \begin{prop}\label{section-4-prop-4.2}
     Let $A$ and  $B$ be two nonempty finite abelian subsets of a right-ordered group $G$ such that the following holds:
     \begin{enumerate}
         \item[\upshape(1)]$\max(A) < \min(B)$,
         \item[\upshape(2)]$\max(A^{2})<\min(B^{2})$,
         \item[\upshape(3)] $\min(B) \notin C_{G}(A)$,
         \item [\upshape(4)] $|A|\geq 4$, $|B|\geq 3$, and $|A\cup B| \geq 7$.
     \end{enumerate} Then $\left|(A\cup B)^{2}\right| \geq  3|A\cup B|-2$.
 \end{prop}
\begin{proof} To get a contradiction, we assume $\left|(A\cup B)^{2}\right| \leq  3|A\cup B|-3$.
  Let $A=\{x_{1}, \ldots, x_{i}\}$ and $B = \{x_{i+1},\ldots,x_{k}\}$ be two sets such that $\langle A \rangle$ and $\langle B \rangle$ are  abelian subgroups of $G$ with $$x_{1}  < \cdots <x_{i}<x_{i+1} < \cdots < x_{k},$$ and $x_{i}^{2}<x_{i+1}^{2}$.  Since $ x_{i+1}\notin C_G(A)$, we have  $x_{i+1}\notin \langle A\rangle$. It follows that  \begin{equation}\label{Sec-3-Eq-2}
      A^{2}\cap (x_{i+1}A\cup Ax_{i+1})=\emptyset.
  \end{equation}
   Note that  \begin{equation}\label{Sec-3-Eq-3}
       \langle B\rangle \leq C_G(x_{i+1}).
    \end{equation}
     Define $$A_{0} := \{x_{t} \in A: x_{t}x_{i+1} \neq x_{i+1}x_{t}\}.$$  Clearly $A_{0} \neq \emptyset$. Let $x_{j} = \max(A_{0})$. Then \begin{equation}\label{Sec-3-Eq-4}
        x_{i+1}x_{j}\neq x_{j}x_{i+1}, \end{equation}
     and if $x_{a}\in A$ with $x_{a}>x_{j}$, then   \begin{equation}\label{Sec-3-Eq-5}       x_{a}\in C_G(x_{i+1}).    \end{equation}
      From \eqref{Sec-3-Eq-3} and \eqref{Sec-3-Eq-4}, it is easy to observe that \begin{equation}\label{Sec-3-Eq-6}
         B^{2}\cap (x_{j}B\cup Bx_{j})=\emptyset.
     \end{equation} 
     
     \textbf{Claim 1.} $B^{2}\cap (x_{i+1}A\cup Ax_{i+1})=\emptyset.$
     
     Suppose that there exists $x_{a} \in A$ such that $x_{i+1}x_{a}\in B^2$. Then there exist $x_{b_1},x_{b_2}\in B$ such that $$x_{i+1}x_{a}=x_{b_1}x_{b_2}.$$ This implies $x_{a} = x_{i+1}^{-1}x_{b_{1}}x_{b_{2}} = x_{b_1}x_{i+1}^{-1}x_{b_2} \geq x_{b_{2}}$, which is not possible. Thus $x_{i+1}A \cap B^2 = \emptyset$. Similarly, we can show that  $Ax_{i+1}\cap B^{2}=\emptyset$. Hence,
     \begin{equation}\label{Section-4-EQU-1}
          B^{2}\cap (x_{i+1}A\cup Ax_{i+1})=\emptyset.
     \end{equation} 

     \textbf{Claim 2.} $A^{2}\cap (x_{j}B\cup Bx_{j})=\emptyset.$
     
     Suppose $x_{j}x_{b}\in A^2$ for some $x_{b}\in B$. Then there exist $x_{a_1}, x_{a_2}\in A$ such that $$x_{j}x_{b}=x_{a_1}x_{a_2}.$$ This implies $x_{b} = x_{j}^{-1}x_{a_{1}}x_{a_{2}} \in \langle A \rangle$. Thus, $x_{j}x_{b}=x_{b}x_{j}$. Now, consider the following cases.
     \begin{enumerate}
         \item If $x_{a_1}>x_{j}$ and $x_{a_2}>x_{j}$, then \eqref{Sec-3-Eq-5} implies that $x_{j} = x_{b}^{-1}x_{a_{1}}x_{a_{2}}\in \langle x_{a_1},x_{a_2},x_{b}\rangle\leq C_G(x_{i+1})$, which contradicts \eqref{Sec-3-Eq-4}. 
         
         \item  If $x_{a_1}\leq x_{j}$, then $x_{a_2}=x_{a_1}^{-1}x_{j}x_{b}=x_{j}x_{a_1}^{-1}x_{b} \geq x_{b}$, which is not possible. 
         
         \item If $x_{a_2}\leq x_{j}$, then $x_{a_1}=x_{a_2}^{-1}x_{j}x_{b}=x_{j}x_{a_2}^{-1}x_{b}\geq x_{b}$,  which is again not possible.
     \end{enumerate} Thus, $x_{j}x_{b} \notin A^{2}$ for any $x_{b} \in B$. Similarly, we have $x_{b}x_{j} \notin A^{2}$ for any $x_{b} \in B$. Therefore, 
    \begin{equation}\label{Prop 4.2 Eq-1..1}
          A^{2}\cap (x_{j}B\cup Bx_{j})=\emptyset.
     \end{equation}
     
\noindent Since $A$ and $B$ are abelian with $\max(A^{2})<\min (B^{2}),$ we have \begin{equation}\label{Sec-3-Eq-7}
			A^{2}\cap B^{2}=\emptyset.
		\end{equation} 

        \textbf{Claim 3.} $Ax_{i+1}\cap x_{j} B=\{x_{j}x_{i+1}\}$.
        
		Suppose $x_{a}x_{i+1}=x_{j}x_{b}$  for some  $x_{a}\in A$ and $x_{b}\in B$.  Consider the following cases: 
		\begin{enumerate}
			\item If $x_{a} > x_{j},$ then it follows from \eqref{Sec-3-Eq-3} and \eqref{Sec-3-Eq-5} that 
			$$ x_{j} = x_{a}x_{i+1}x^{-1}_{b} \in \langle x_{a}, x_{i+1}, x_{b}\rangle\leq C_G(x_{i+1}),$$
			which is a contradiction to \eqref{Sec-3-Eq-4}.
			\item If $x_{a}<x_{j}$, then
			$$x_{i+1} = x^{-1}_{a}x_{j}x_{b} =x_{j}x^{-1}_{a}x_{b} > x_{b},$$ which is a contradiction because  
			$ x_{i+1}$ is the smallest element of  $B$.
		\end{enumerate}
		Thus, $x_{a} = x_{j}$, and this implies $x_{i+1} = x_{b}$. Therefore, \begin{equation*}
			Ax_{i+1}\cap x_{j} B=\{x_{j}x_{i+1}\}.
		\end{equation*}  This implies that \begin{equation}\label{Sec-3-Eq-9}
			|Ax_{i+1}\cup x_{j}B|=i+k-i-1=k-1.
		\end{equation}
		 By Proposition \ref{Prop-3.1}, \eqref{Sec-3-Eq-2}, 
         \eqref{Sec-3-Eq-6}, \eqref{Section-4-EQU-1},  \eqref{Prop 4.2 Eq-1..1}, \eqref{Sec-3-Eq-7}, and \eqref{Sec-3-Eq-9}, we have
		\begin{align*}
			3k-3 \geq |(A\cup B)^2| &\geq |A^{2}\cup B^{2}\cup Ax_{i+1}\cup x_{j}B|\\
			& = |A^{2}|+|B^{2}|+|Ax_{i+1}\cup x_{j}B|\\
			&\geq 2i-1+2(k-i)-1+k-1=3k-3.
		\end{align*} Therefore, 
   $$ (A\cup B)^{2}= A^{2}\cup B^{2}\cup Ax_{i+1}\cup x_{j}B,$$
 with  $|A^{2}| = 2i-1 \text{ and } |B^{2}| = 2(k-i)-1$.  By Proposition \ref{Prop-3.1}, we obtain $$A = \{x, xg,xg^{2},\ldots, xg^{i-1}\} \text{ and } B=\{y, yt, yt^{2}, \ldots, yt^{k-i-1}\},$$ where $xg=gx$ and $ty=yt$, with $g\neq e$ and $t\neq e$. Therefore, the elements of the set $A$ can be arranged in one of the following ways:
 $$ x<xg<\cdots<xg^{i-1}$$
 or $$xg^{i-1}<xg^{i-2}<\cdots<x,$$
 and the elements of the set $B$ can be arranged in one of the following ways:
 $$y<yt<\cdots<yt^{k-i-1}$$
or  $$yt^{k-i-1}<yt^{k-i-2}<\cdots<y,$$

 Observe that if $j+2 \leq i$, then $x_{j+1}, x_{j+2} \in C_{G}(x_{i+1})$, and by Lemma \ref{Lemma 4.1}, we have $x_{i+1}\in C_{G}(A)$, which is a contradiction. Therefore, 
$$j\leq i \leq j+1 \ \text{and} \ x_{i}\in \{x_{j+1}, x_{j}\}.$$    

 Note that $x_{i+1}x_{j}\in (A\cup B)^{2}$, but $x_{i+1}x_{j} \notin A^{2} \cup B^{2}$. If $x_{i+1}x_{j}\notin Ax_{i+1} \cup x_{j}B$, then $|(A\cup B)^{2}|\geq 3k-2$, which is a contradiction. So, $x_{i+1}x_{j} \in Ax_{i+1} \cup x_{j}B$. 
 
 First, we assume that $x_{i+1}x_{j} \in x_{j}B$. There exists $x_{b}\in B$  with $x_{b}>x_{i+1}$ such that 
  \begin{equation}\label{thm-2.5-0}
     x_{i+1}x_{j}=x_{j}x_{b}.
     \end{equation}  
Now, consider the following cases. \par

\textbf{Case 1} ($x_{j}x_{i+1}<x_{j}x_{i+2}$). Since $B$ is a geometric progression and $x_{j}x_{i+1}<x_{j}x_{i+2}$, we have \begin{equation}\label{thm-2.2-1.1}
    x_{j}x_{i+1}<x_{j}x_{i+2}<\cdots<x_{j}x_{k}.\end{equation}
Also, we have \begin{equation}\label{thm-2.5-1}
    x_{i+1}x_{j}<x_{i+2}x_{j}<\cdots<x_{k}x_{j}. 
\end{equation} 
Note that if $|Bx_{j} \setminus x_{j}B| =0$, then $Bx_{j} = x_{j}B$ because $|Bx_{j}| = |x_{j}B|$, and it follows from \eqref{thm-2.2-1.1} and  \eqref{thm-2.5-1} that $x_{i+1}x_{j} = x_{j}x_{i+1}$, which contradicts \eqref{Sec-3-Eq-4}. Therefore,  $|Bx_{j} \setminus x_{j}B| \geq 1$. Thus, there exists $x_{c}\in B$ such that $x_{c}>x_{i+1}$ and $x_{c}x_{j}\notin x_{j}B$. Clearly, $x_{c}x_{j}\notin A^{2} \cup B^{2}$. If $x_{c}x_{j}\notin Ax_{i+1}$, then $|(A\cup B)^{2}|\geq 3k-2$, which is a contradiction. Thus,  $x_{c}x_{j}\in Ax_{i+1}$. This implies that there exists $x_{a}\in A$ such that \begin{equation}\label{thm-2.5-2}
    x_{c}x_{j} =x_{a}x_{i+1}.
\end{equation} Since $x_{c}x_{j}\notin x_{j}B$, we must have $x_{a}\neq x_{j}$. If $x_{a}>x_{j}$, then  $x_{j}= x_{c}^{-1}x_{a}x_{i+1}\in \langle x_{c}, x_{i+1},x_{a}\rangle \leq C_G(x_{i+1})$, contradiction to \eqref{Sec-3-Eq-4}. So, $x_{a}<x_{j}$. Since $\langle B\rangle$ is abelian, it follows from \eqref{thm-2.5-2} and \eqref{thm-2.5-0} that $$x_{i+1}x_{a}x_{i+1}=x_{i+1}x_{c}x_{j}=x_{c}x_{i+1}x_{j}=x_{c}x_{j}x_{b}.$$
This gives $x_{a}x_{i+1}=x_{i+1}^{-1}x_{c}x_{j}x_{b}=x_{c}x_{i+1}^{-1}x_{j}x_{b}>x_{j}x_{b}$.
But \eqref{thm-2.2-1.1} implies that $x_{a}x_{i+1}<x_{j}x_{i+1}<x_{j}x_{b}$, which is a contradiction. Therefore, this case is not possible.\\

\textbf{Case 2} ($x_{j}x_{i+1}>x_{j}x_{i+2} \text{ and } x_{i+1}x_{1}>x_{i+1}x_{2}$).  Since $A$ is a geometric progression, we have   \begin{equation}\label{thm-2.5-0.1.1}
    x_{i+1} x_{1} > x_{i+1}x_{2}> \cdots > x_{i+1} x_{j}>\cdots > x_{i+1} x_{i}.
\end{equation} 
 Since $j\leq i \leq j+1$ and $|A|\geq 4$, there exists $x_{a}\in A$ such that \begin{equation}\label{Proposition-4.2-Eq-6.1}
    x_{a} < x_{j}.
\end{equation}  Clearly $x_{i+1} x_{a}\notin A^2\cup B^2$. If $x_{i+1}x_{a}\notin Ax_{i+1}\cup x_{j}B$, then $|(A\cup B)^{2}|\geq 3k-2$, which is a contradiction. Suppose $x_{i+1} x_{a} \in Ax_{i+1}\cup x_{j}B$. Consider the following subcases.\\
\textbf{Subcase 2.1} $(x_{i+1} x_{a} \in A x_{i+1})$.  There exists $x_{a'}\in A$ such that $x_{i+1}x_{a}=x_{a'}x_{i+1}$. Note that if $x_{j}<x_{a'}$, then from \eqref{Sec-3-Eq-5}, we have $x_{i+1}x_{a}=x_{a'}x_{i+1} = x_{i+1}x_{a'}$, and this gives $x_{a'}=x_{a}<x_{j}$, which is contradiction. Therefore $x_{a'}\leq x_{j}$.   Since $x_{i+1}x_{j}=x_{j}x_{b}$, we write $$x_{a'}x_{i+1}x_{j}=x_{i+1}x_{a}x_{j}=x_{i+1}x_{j}x_{a}=x_{j}x_{b}x_{a}.$$ This implies that $x_{i+1}x_{j}=x_{j}x_{a'}^{-1}x_{b}x_{a}\geq x_{b}x_{a}$. But \eqref{thm-2.5-0.1.1} implies that $x_{i+1}x_{a}>x_{i+1}x_{j}\geq x_{b}x_{a}$ and hence $x_{i+1}>x_{b}$, which is not possible.\\
\textbf{Subcase 2.2} $(x_{i+1} x_{a} \in x_{j}B)$. We can write $x_{i+1} x_{a} = x_{j} x_{b'}$ for some $x_{b'} \in B$.  Multiply \eqref{thm-2.5-0} by $x_{b'}$ on the right side, we get
   $$x_{i+1} x_{j} x_{b'} = x_{j} x_{b} x_{b'} = x_{j} x_{b'} x_{b} = x_{i+1} x_{a} x_{b}.$$   This implies $x_{i+1} x_{a}=x_{j} x_{b'} = x_{a} x_{b}$.   Since $x_{i+1} x_{a} > x_{i+1} x_{j}$ and $x_{i+1} x_{j} = x_{j} x_{b}$,  
    we get $x_{a} x_{b} > x_{j} x_{b}$.  
    This gives $x_{a} > x_{j}$, which contradicts \eqref{Proposition-4.2-Eq-6.1}.\\

\textbf{Case 3} ($x_{j}x_{i+1}>x_{j}x_{i+2} \text{ and } x_{i+1}x_{1}<x_{i+1}x_{2}$). In this case, we have \begin{equation}\label{subcase-3.2.2}
    x_{i+1}x_{1}<x_{i+1}x_{2},<\cdots<x_{i+1}x_{i},
    \end{equation} \begin{equation}\label{subcase-3.2.3}
        x_{1}x_{i+1}<x_{2}x_{i+1}<\cdots<x_{i}x_{i+1},
    \end{equation}
  and
    \begin{equation}\label{B}
        x_{j}x_{i+1}>x_{j}x_{i+2}>\cdots x_{j}x_{k}
    \end{equation} because $A$ and $B$ are geometric progression.  Since $j\leq i \leq j+1$ and  $|A|\geq 4$, there exists $x_{a}\in A$ with  $x_{a}<x_{j}$. Clearly $x_{i+1}x_{a}\notin A^{2}\cup B^{2}$, and hence $x_{i+1}x_{a}\in Ax_{i+1}\cup x_{j}B$. Suppose that $x_{i+1} x_{a} \in x_{j}B$. Then there exists some $x_{b'} \in B$ such that  
$$x_{i+1} x_{a} = x_{j} x_{b'}.$$ By \eqref{subcase-3.2.2}, we have  $x_{j}x_{b'}=x_{i+1} x_{a} < x_{i+1}x_{j}=x_{j}x_{b}$. Therefore, by \eqref{B}, we have $x_{b} < x_{b'}$. Also, we have 
$$
x_{j} x_{b} x_{a} =x_{i+1} x_{j} x_{a}=x_{i+1} x_{a} x_{j} = x_{j} x_{b'} x_{j},
$$
and this gives
$x_{b} x_{a} = x_{b'} x_{j}$.
Since $x_{b} < x_{b'}$, we get $x_{a} = x_{b'} x_{b}^{-1} x_{j}>x_{j}$, which contradicts the assumption that $x_a<x_j$. Therefore, $x_{i+1} x_{a} \notin x_{j}B$ for all $x_{a}<x_{j}$, and hence,
$$x_{i+1} x_{a} \in Ax_{i+1}\text{ for all } x_{a}\in A \text{ with } x_{a}\neq x_{j}.$$ This implies \begin{equation}\label{Equation-101}
    Ax_{i+1} \cup x_{i+1}A = Ax_{i+1} \cup \{x_{i+1}x_{j}\} \ \text{and} \  |Ax_{i+1} \cup x_{i+1}A| = |A|+1.
\end{equation} Since $|A|\geq 4$ and $j\leq i \leq j+1$, there exist $x_{j-1}$, $x_{j-2} \in A$ such that $x_{j-2}<x_{j-1}<x_{j}$ and $\{x_{i+1}x_{j-2},x_{i+1}x_{j-1}\}\subset Ax_{i+1}$. Therefore, we have \begin{equation}\label{subcase-3.2.1}
    x_{i+1}x_{j-2}=x_{a_1}x_{i+1}\text{ and }x_{i+1}x_{j-1}=x_{a_2}x_{i+1} \end{equation} for some $x_{a_1}, x_{a_2}\in A$. Since $x_{i+1}x_{j-2}<x_{i+1}x_{j-1}<x_{i+1}x_{j}=x_{j}x_{b}<x_{j}x_{i+1}$, we have $$x_{a_1}<x_{a_{2}}<x_{j}.$$
Let, if possible, $x_{a_2}<x_{j-1}$, then using  \eqref{subcase-3.2.2}, \eqref{subcase-3.2.3}, and \eqref{subcase-3.2.1}, we have $$x_{i+1}x_{1}<\ldots<x_{i+1}x_{j-2}<x_{i+1}x_{j-1} = x_{a_{2}}x_{i+1}<x_{j-1}x_{i+1}<\ldots<x_{i}x_{i+1}.$$ It is easy to observe that $x_{i+1}x_{j}$ is distinct from the above-listed elements; therefore, $|Ax_{i+1} \cup x_{i+1}A|\geq |A|+2$, which contradicts \eqref{Equation-101}. Thus, $x_{j-1} \leq x_{a_{2}} < x_{j}$. Hence $x_{a_{2}}= x_{j-1}$ and $x_{i+1}x_{j-1} = x_{j-1}x_{i+1}$. With a similar argument, we obtain $x_{i+1}x_{j-2}=x_{j-2}x_{i+1}$. Using Lemma \ref{Lemma 4.1}, we obtain $x_{i+1}\in C_G(A)$, which contradicts \eqref{Sec-3-Eq-4}. 

Therefore, from the above discussion, we observe that $x_{i+1}x_{j}\notin x_{j}B$. Now, suppose that $x_{i+1}x_{j}\in  Ax_{i+1}$. Then there exist  $x_{a}\in A$ with $x_{a}<x_{j}$  such that  \begin{equation}\label{thm-2.5-3} x_{i+1}x_{j}=x_{a}x_{i+1}.
   \end{equation}
     
 Now, consider the following cases. \par
     \textbf{Case 1} ($x_{j}x_{i+1}<x_{j}x_{i+2}\text{ and }x_{i+1}x_{1}<x_{i+1}x_{2}$). Since $A$ is a geometric progression, we have
  \begin{equation}\label{thm-2.5-4}
    x_{i+1}x_{1}<x_{i+1}x_{2}<\cdots<x_{i+1}x_{i}.
\end{equation}  
 Also, we  have \begin{equation*}\label{thm-2.5-4-4}
     x_{1}x_{i+1}<x_{2}x_{i+1}<\cdots<x_{i}x_{i+1}.
 \end{equation*}
 Note that $|Ax_{i+1}\setminus x_{i+1}A | \geq 1$, otherwise, we have $Ax_{i+1} = x_{i+1}A$, which implies $x_{i+1} \in C_{G}(A)$. Therefore, there exists $x_{a'}<x_{j}\in A$ such that $x_{i+1}x_{a'}\notin Ax_{i+1}$. Clearly $x_{i+1}x_{a'}\notin A^{2}\cup B^{2}$. If $x_{i+1}x_{a'}\notin x_{j}B$, then clearly $|(A\cup  B)^{2}|\geq 3k-2$, which is a contradiction. Therefore,   $x_{i+1}x_{a'}\in x_{j}B$. This implies that there exist $x_{b'}\in B$ such that 
     $x_{i+1}x_{a'}=x_{j}x_{b'}$. Since $\langle A\rangle$ is abelian, it follows that $$x_{j}x_{b'}x_{j}=x_{i+1}x_{a'}x_{j}=x_{i+1}x_{j}x_{a'}=x_{a}x_{i+1}x_{a'}.$$ Since $x_{j}>x_{a}$, we have  $x_{i+1}x_{a'} = x_{a}^{-1}x_{j}x_{b'}x_{j}=x_{j}x_{a}^{-1}x_{b'}x_{j}>x_{b'}x_{j}$. But $x_{a'}<x_{j}$ and $x_{i+1}\leq x_{b'}$, and from \eqref{thm-2.5-4}, we obtain $x_{i+1}x_{a'}<x_{i+1}x_{j}\leq x_{b'}x_{j}$, which is a contradiction.  Thus, this case is not possible.
 
 \textbf{Case 2} ($ x_{j}x_{i+1}<x_{j}x_{i+2}\text{ and }x_{i+1}x_{1}>x_{i+1}x_{2}$). In this case, we have  \begin{equation}\label{thm-2.5-case-5-eq-2}
    x_{i+1}x_{1}>x_{i+1}x_{2}>\cdots>x_{i+1}x_{j}>\cdots>x_{i+1}x_{i},
\end{equation}
\begin{equation}\label{thm-2.5-case-5-eq-3}
     x_{i+1}x_{j}<x_{i+2}x_{j}<\cdots<x_{k}x_{j},
\end{equation}
and
\begin{equation}\label{thm-2.5-case-5-eq-4}
  x_{j}x_{i+1}<x_{j}x_{i+2}<\cdots<x_{j}x_{k}.  
\end{equation}
Suppose $x_{b}x_{j}\in Ax_{i+1} \text{ for some } x_{b}\in B \text{ with } x_b>x_{i+1}$. There exist some $x_{a'}\in A$ such that  $$x_{b}x_{j}=x_{a'}x_{i+1}.$$
Since $x_{a}x_{i+1}=x_{i+1}x_{j}<x_{b}x_{j}=x_{a^{''}}x_{i+1}$, we have $x_{a}<x_{a'}$. This, together with \eqref{thm-2.5-case-5-eq-2} gives us $x_{i+1}x_{a'}<x_{i+1}x_{a}$. Note  that $$x_{b}x_{i+1}x_{j}=x_{i+1}x_{b}x_{j} = x_{i+1}x_{a'}x_{i+1}<x_{i+1}x_{a}x_{i+1} = x_{i+1}x_{i+1}x_{j}.$$ This gives $x_{b}<x_{i+1}$, which is not possible because $x_{b}$ is the smallest element of the set $B$. Therefore, we have $$x_{b}x_{j}\in x_{j}B \text{ for all } x_{b}\in B \text{ with } x_b>x_{i+1}.$$ 
This implies \begin{equation*}
|x_{j}B\cup Bx_{j}|= |x_{j}B\cup \{x_{i+1}x_{j}\}|= |B|+1.\end{equation*}
Since $|B|\geq 3$, using \eqref{thm-2.5-case-5-eq-3} and \eqref{thm-2.5-case-5-eq-4}, we have $x_{i+2}x_{j}\in\{x_{j}x_{i+1},x_{j}x_{i+2}\}$ and $x_{i+3}x_{j} \in \{x_{j}x_{i+1},x_{j}x_{i+2}, x_{j}x_{i+3}\}$. If $x_{i+2}x_{j}=x_{j}x_{i+1}$, 
then $$x_{a}x_{i+1}=x_{i+1}x_{j}<x_{i+2}x_{j}=x_{j}x_{i+1}.$$This gives $x_{a}<x_{j}$, using  \eqref{thm-2.5-case-5-eq-2}, we get $x_{i+1}x_{j}<x_{i+1}x_{a}$.  Thus,  $$x_{i+2}x_{i+1}x_{j}=x_{i+1}x_{i+2}x_{j}=x_{i+1}x_{j}x_{i+1}< x_{i+1}x_{a}x_{i+1}= x_{i+1}x_{i+1}x_{j}.$$ This gives $x_{i+2}<x_{i+1}$, which is not possible. Therefore $x_{i+2}x_{j}= x_{j}x_{i+2}$. Since $x_{i+3}x_{j}\in x_{j}B$ and $x_{i+2}x_{j}<x_{i+3}x_{j}$, we get $x_{i+3}x_{j}=x_{j}x_{i+3}$. Therefore, by Lemma  \ref{Lemma 4.1}, we obtain $x_{j}\in C_G(B)$, which is a contradiction to  \eqref{Sec-3-Eq-4}. 

\textbf{Case 3} ($x_{j}x_{i+1}>x_{j}x_{i+2}\text{ and }x_{i+1}x_{1}<x_{i+1}x_{2}$). In this case, we have 
\begin{equation}\label{thm-2.5-case-6-eq-4}
    x_{i+1}x_{1}<x_{i+1}x_{2}<\cdots<x_{i+1}x_{i},
\end{equation}  
 and \begin{equation}\label{thm-2.5-case-6-eq-5}
     x_{1}x_{i+1}<x_{2}x_{i+1}<\cdots<x_{i}x_{i+1}.
 \end{equation} It is easy to observe that  $Ax_{i+1} \neq x_{i+1}A$, otherwise, \eqref{thm-2.5-case-6-eq-4} and \eqref{thm-2.5-case-6-eq-5} implies that $x_{i+1} \in C_{G}(A)$. Therefore, there exists $x_{a'}<x_{j}\in A$ such that $x_{i+1}x_{a'}\notin Ax_{i+1}$. Clearly, $x_{i+1}x_{a'}\notin A^{2} \cup B^{2}$, but $x_{i+1}x_{a'}\notin (A\cup B)^{2}$. Therefore,  $x_{i+1}x_{a'}\in x_{j}B$. So, there exist $x_{b'}\in B$ such that 
    $ x_{i+1}x_{a'}=x_{j}x_{b'}$. Since $\langle A\rangle$ is abelian, we have $$x_{j}x_{b'}x_{j}=x_{i+1}x_{a'}x_{j}=x_{i+1}x_{j}x_{a'}=x_{a}x_{i+1}x_{a'}.$$
This gives $x_{i+1}x_{a'}=x_{j}x_{a}^{-1}x_{b'}x_{j}>x_{b'}x_{j}$. But \eqref{thm-2.5-case-6-eq-4} implies that $x_{i+1}x_{a'}<x_{i+1}x_{j}\leq x_{b'}x_{j}$, which is a contradiction.

  \textbf{Case 4} ($x_{j}x_{i+1}>x_{j}x_{i+2}\text{ and }x_{i+1}x_{1}>x_{i+1}x_{2}$).
In this case, we have 
 \begin{equation}\label{thm-2.5-case-7-eq-5}
      x_{j}x_{k}<x_{j}x_{k-1}<\cdots<x_{j}x_{i+1}
  \end{equation}
  because $B$ is a geometric progression. Also, we have
  \begin{equation}\label{thm-2.5-case-7-eq-4}
      x_{i+1}x_{j}<x_{i+2}x_{j}<\cdots<x_{k}x_{j}.
  \end{equation}
  Note that $x_{b}x_{j}\notin A^{2} \cup B^{2}$ for all $x_{b} \in B$. 
   Suppose there exists $x_{b} \in B$ such that 
$x_{b}x_{j}\in Ax_{i+1}$ with  $x_{b}>x_{i+1}.$   This implies \begin{equation}\label{thm-2.5-case-7-eq-6}
    x_{b}x_{j}=x_{a'}x_{i+1} \text{ for some } x_{a'}\in A.
\end{equation}
 Since $x_{a}x_{j} = x_{i+1}x_{j}<x_{b}x_{j} =x_{a'}x_{i+1}$, we have   $x_{a}<x_{a'}$. Using \eqref{thm-2.5-3}, \eqref{thm-2.5-case-7-eq-5}, and \eqref{thm-2.5-case-7-eq-6}, we have  $$x_{j}x_{a}x_{i+1}=x_{j}x_{i+1}x_{j}>x_{j}x_{b}x_{j}=x_{j}x_{a'}x_{i+1}.$$ This implies $x_{j}x_{a}>x_{j}x_{a'}$, and hence $x_{a}>x_{a'}$, which contradicts  $x_{a}<x_{a'}$. Therefore, we obtain that $x_{b}x_{j}\notin Ax_{i+1}$ for all $x_{b} \in B$ with  $x_{b}>x_{i+1}.$ Thus,  $$x_{b}x_{j}\in x_{j}B \text{ for all } x_{b}\in B \text{ with } x_{b}>x_{i+1}.$$ Therefore, $|x_{j}B\cup Bx_{j}|= |B|+1$.
 From \eqref{thm-2.5-case-7-eq-5} and \eqref{thm-2.5-case-7-eq-4}, if $x_{k}x_{j}\leq x_{j}x_{i+3}$, then $$ x_{i+1}x_{j}<x_{i+2}x_{j}<\cdots<x_{k}x_{j}\leq  x_{j}x_{i+3}< x_{j}x_{i+2}< x_{j}x_{i+1},$$ which yields $|x_{j}B\cup Bx_{j}|> |B|+1$. Hence $x_{k}x_{j}\in  \{x_{j}x_{i+2},x_{j}x_{i+1}\}$. If  $x_{k}x_{j}=x_{j}x_{i+1}$, then by \eqref{thm-2.5-case-7-eq-5} and \eqref{thm-2.5-3}, $$x_{j}x_{a}x_{i+1}=x_{j}x_{i+1}x_{j}>x_{j}x_{k}x_{j}=x_{j}x_{j}x_{i+1}.$$ This gives $x_{a}>x_{j}$ , which is not possible. Hence $x_{k}x_{j}=x_{j}x_{i+2}$. Consequently, by \eqref{thm-2.5-case-7-eq-5} and \eqref{thm-2.5-case-7-eq-4}, we obtain  
  \begin{equation}\label{thm-2.5-case-7-eq-7}
    x_{i+2}x_j = x_jx_k \quad \text{and} \quad x_{i+3}x_j = x_jx_{k-1}.
  \end{equation}
As $B=\{y,yt,yt^2,\ldots,yt^{k-i-1}\}$, where $ty=yt$ and $t\neq e$ for some elements $y,t\in G$, we have the following cases.
\begin{enumerate}
    \item If $t<e$, then $yt^{k-i-1}<yt^{k-i-2}<\cdots<yt<y$,  and hence $x_{i+1}=yt^{k-i-1}=tx_{i+2}$. From \eqref{thm-2.5-case-7-eq-7}, we get $tx_{j}t=x_{j}$. Now, using 
$x_{i+2}x_{j}=x_{j}x_{k}$, we have $$x_{j}x_{k}=x_{i+2}tx_{j}t=x_{i+1}x_{j}t=x_{a}x_{i+1}t,$$ which implies $x_{i+1}t>x_{k}$, and hence $t^{k-i}>e$, which contradicts our assumption $t<e$.
\item  If $t>e$, then $y<yt<yt^{2}<\cdots<yt^{k-i-1}$. Again using \eqref{thm-2.5-case-7-eq-7}, we get $tx_{j}t=x_{j}$. Since $x_{j}x_{i+1}>x_{j}x_{k}$, we obtain $$x_{j}x_{a}x_{i+1}=x_{j}x_{i+1}x_{j}>x_{j}x_{k}x_{j}=x_{j}^{2}x_{i+2}=x_{j}^{2}tx_{i+1}.$$ This implies $x_{j}x_{a}>x_{j}^{2}t$. Since  $x_{a}tx_{j}t=x_{a}x_{j}=x_{j}x_{a}>x_{j}^{2}t$, we get $x_{a}t>x_{j}$. Therefore $x_{a}t>tx_{j}t$, and  this gives $x_{a}>tx_{j}$. But  $tx_{j}>x_{j}>x_{a}$, which is a contradiction.
\end{enumerate}
Therefore, $x_{i+1}x_{j}\notin Ax_{i+1}\cup x_{j}B$. Hence, we obtain $|(A\cup B)^{2}|\geq 3k-2$. This completes the proof.
\end{proof}

\begin{proof}[Proof of Theorem \ref{Theorem-2}] Without loss of generality, we assume that $\min(B)\notin C_G(A)$. Under this assumption, the result follows immediately from Proposition \ref{section-4-prop-4.1} and Proposition \ref{section-4-prop-4.2}.
\end{proof}
          

 \section{Proof of Theorem \ref{Theorem-3}}\label{section-6}
          
          The Baumslag-Solitar groups $\mathrm{BS}(p,q)$ are two-generated groups with one relation defined as follows:$$\mathrm{BS}(p,q):=\langle a,b : ab^{p}=b^{q}a\rangle,$$ where $p$ and  $q$ are integers. Both direct and inverse problems have been studied extensively in these groups (see \cite{Freiman_Herzog_2014, Freiman_et_al_2015, Chahal_Kaur_2025}).  In this section, we focus on inverse problems for the $\mathrm{BS}(1,q)$ group.  Since the group $\mathrm{BS}(1,q)$ is isomorphic to  $\mathbb{Z}[1/q]\rtimes\mathbb{Z}$, where  $\mathbb{Z}[1/q]\rtimes\mathbb{Z}$ is a right-ordered  group under the operation $$(r,n)(s,m) = (r + q^{n}s,n + m), ~\text{for}~r,s\in \mathbb{Z}[1/q]~\text{and}~n,m\in \mathbb{Z}.$$  Therefore, the following proposition also holds for the  Baumslag-Solitar group $\mathrm{BS}(1,q)$ when $q<-1$.
          
\begin{prop}\textup{\cite[Lemma 4.4]{Neetu_Mohan_Shankar}}\label{Sec-5-Lemma-1} Let $q< -1$ be an integer. Let $S$ be a nonempty finite subset of a right ordered group $\mathbb{Z}[1/q]\rtimes\mathbb{Z}$ with $|S|\geq 3$ and $\min (S)=e$. Let $m= \max (S)$ and $A = S \setminus \{m\}$. Then either 
     $$ \langle S\rangle ~\text{is an abelian subgroup of $G$,}$$
     or  $$ \left|A^{2}\right| \leq \left|S^{2}\right|-3.$$
   
	\end{prop}
   \begin{prop}\textup{\cite[Theorem 3.5]{Neetu_Mohan_Shankar}}\label{NM-k=3}
       	Let $S$ be a nonempty subset of a right-ordered group $G$  with $|S|=3$ and $\min(S)=e$. If $\left|S^{2}\right| \leq 5$, then $\langle S \rangle$ is an abelian  subgroup of $G$.
   \end{prop}
\begin{prop}\textup{\cite[Theorem 3.6]{Neetu_Mohan_Shankar}}\label{NM-k=4}
		Let $S$ be a nonempty subset of a right-ordered group $G$  with $|S|=4$ and $\min(S)=e$. If $\left|S^{2}\right| \leq 8$, then $\langle S \rangle$ is an abelian  subgroup of $G$.
	\end{prop}
\begin{prop}\label{lemma 1-theorem-2.13}
      Let $S$ be a nonempty finite subset of right ordered group $G$ with  $\left|S\right| =3$, $e =\min(S)$ and  $ \left|S^{2}\right| = 6$. Then either $\langle S\rangle$ is abelian
   or  $S=\{e, x, xg\}$ with $e<x<xg$, $xg\neq gx$ and $(xg)^{2}=x^{2}$.
\end{prop}

\begin{proof}
     Let $S=\{x_{1},x_{2}, x_{3}\}$ with $e=x_{1}<x_{2}<x_{3}$. If $x_{2}x_{3} = x_{3}x_{2}$, then $\langle S \rangle$ is abelian. Assume that $x_{2}x_{3} \neq  x_{3}x_{2}$. Then $x_{2}^{2} \neq x_{3}$. Note that  $$e<x_{2}<x_{3}<x_{2}x_{3}<x_{3}^{2}.$$ It is easy to observe that  $x_{3}x_{2}$ is not equal to any of the above-listed  elements. Therefore, $$S^{2} = \{e,x_{2},x_{3}, x_{2}x_{3}, x_{3}x_{2},x_{3}^{2}\}.$$ Since $x_{2}^{2} \in S^{2}$, we have $x_{2}^{2} = x_{3}^{2}$. This gives $S = \{e,x,xg\},$ where $x=x_{2}$ and $g=x_{2}x_{3}^{-1}$.
\end{proof}
\begin{prop} \label{lemma 2-theorem-2.13}
Let $q<-1$ be an integer.
     Let $S$ be a nonempty finite subset of $\mathrm{BS}(1,q)$ with  $\left|S\right| =4$, $e =\min(S)$, and  $ \left|S^{2}\right| = 9$.  Then there exist $x, g \in \mathrm{BS}(1,q)$ such that one of the following holds: 
   \begin{enumerate}
    \item[\upshape(i)] $\langle S\rangle$ is abelian,
    \item [\upshape(ii)]  $S=\{e, x, xg, (xg)^{2}\}$ with $e<x<xg<(xg)^{2}$, $xg\neq gx$, and  $x^{2}=(xg)^{2}$,
    \item [\upshape(iii)] $S=\{e, x, xg, xg^{2}\}$ with $e<x<xg<xg^{2}$, $xg\neq gx$, and $(xg^{2})^{2}=(xg)^{2}=x^{2}$.
\end{enumerate}
\end{prop}
\begin{proof}
If $\langle S \rangle$ is abelian, then we are done. Let $\langle S \rangle$ be a nonabelian subgroup. Let $S=\{x_{1},x_{2},x_{3},x_{4}\}$ and $A=\{x_{1},x_{2},x_{3}\}$ with $e=x_{1}<x_{2}<x_{3}<x_{4}$. By Proposition \ref{Sec-5-Lemma-1}, we have $|A^{2}|\leq 6$.  If $\langle A\rangle$ is abelian, then by  Proposition \ref{prop-2}, we obtain that  $S=\{e,g,g^{2}, x_{4}\}$,  with $[g^{i},x_{4}]=g^{2-2i}$ for $i\in  [\![ 0,2 ]\!]$. But for $i = 0$, we get that $[e,x_{4}]=g^{2}$, and this implies $g^2=e$, which is not possible. Therefore, $\langle A\rangle$ is nonabelian. So
\begin{equation}\label{Lemma-5.5-Eq-1}
    x_{2}x_{3} \neq x_{3}x_{2}
    \end{equation} Using Theorem \ref{NM-k=3}, we obtain that $|A^{2}|\geq 6$. Thus $|A^{2}|=6$. Applying  Lemma \ref{lemma 1-theorem-2.13}, we get $A=\{e,x,xg\}$ with $e<x<xg$, $xg\neq gx$ and $(xg)^{2}=x^{2}$. Therefore, $x_{2}=x$ and $x_{3}=xg$. This gives $x_{3}^{2}=x_{2}^{2}$ and $g=x_{2}^{-1}x_{3}=x_{2}x_{3}^{-1}<e$. By right ordering, we have 
    $$e<x_{2}<x_{3}<x_{2}x_{3}<x_{3}^{2}=x_{2}^{2}<x_{3}x_{2},$$
 $$A^{2}=\{e,x_{2}, x_{3}, x_{2}x_{3}, x_{2}^{2}, x_{3}x_{2} \}.$$ 
 If $x_{4}x_{2}< x_{2}x_{4}$, then  $\max(A^{2})= x_{3}x_{2}<x_{4}x_{2}<x_{2}x_{4}<x_{3}x_{4}<x_{4}^{2}$. This gives  
 $$|S^{2}|\geq A^{2} + |\{x_{4}x_{2}, x_{2}x_{4}, x_{3}x_{4}, x_{4}^{2}\}|=10,$$ which is a contradiction, thus we have $x_{2}x_{4}\leq x_{4}x_{2}$.\par 
 \textbf{Case 1} $(x_{4}x_{2}= x_{2}x_{4})$. Since  $\max(A^{2})= x_{3}x_{2}<x_{4}x_{2}= x_{2}x_{4}<x_{3}x_{4}<x_{4}^{2}$, we have  $$S^{2}= A^{2}\cup \{x_{2}x_{4}, x_{3}x_{4}, x_{4}^{2}\}.$$  Further, $x_{4}\in S^{2}$ implies that $x_{4}\in \{x_{2}x_{3}, x_{2}^{2}, x_{3}x_{2}\}$.   If $x_{4}=x_{2}x_{3}$, then $x_{2}x_{3}x_{2}=x_{4}x_{2}=x_{2}x_{4}$, implies $x_{3}x_{2}=x_{4}$, which contradicts \eqref{Lemma-5.5-Eq-1}. Thus, $x_{4}\neq x_{2}x_{3}$. A similar argument also implies that $x_{4} \neq x_{3}x_{2}$. Therefore $x_{4}=x_{2}^{2}$, and hence  (ii) holds. \par 
\textbf{Case 2} $(x_{2}x_{4}<x_{4}x_{2} \ \text{and} \ x_{2}x_{3}<x_{2}x_{4})$. Clearly, $x_{4}x_{2} \notin A^{2}$. Further, if $x_{2}x_{4}\in A^{2}$, then $x_{2}x_{4}=x_{3}x_{2}=x_{2}gx_{2}$. This implies $x_{4}=gx_{2}<x_{2}$, which is a contradiction. Therefore, $x_{2}x_{4}\notin A^{2}$. Let $X=A^{2}\cup \{x_{2}x_{4}, x_{4}x_{2}\}$. Now, we show that $x_{3}x_{4}\notin X$.

Suppose $x_{3}x_{4}\in X$. Then $x_{3}x_{4}=x_{4}x_{2}$. This implies $\max(A^{2})<x_{4}x_{2}<x_{4}^{2}$. Thus $S^{2}=X\cup \{x_{4}^{2}\}$. Since $x_{4}\in S^{2}$, we have  $x_{4}\in \{x_{2}x_{3}, x_{3}^{2}, x_{3}x_{2}\}$.
\begin{enumerate}
    \item If $x_{4}=x_{2}x_{3}$, then $$x_{2}x_{2}gx_{2}=x_{2}x_{3}x_{2}=x_{4}x_{2}=x_{3}x_{4}=x_{2}gx_{4}=x_{2}x_{2}x_{3}^{-1}x_{4}.$$ This implies $$gx_{2}=x_{3}^{-1}x_{4}=x_{3}^{-1}x_{2}x_{3}=x_{3}^{-1}x_{2}x_{2}g.=x_{3}^{-1}x_{3}^{2}g=x_{3}g.$$ Further, we have $$x_{2}= x_{2}x_{3}^{-1}x_{3}=gx_{3}=gx_{2}g=x_{3}g^{2}=x_{2}g^{3}.$$ This gives $g^{3}=e$, which is not possible.
    \item If $x_{4}=x_{3}^{2}$, then using $x_{4}x_{2}=x_{3}x_{4}$, we obtain $x_{4}=x_{3}x_{2}$, which is not possible.
    
    \item    If $x_{4}=x_{3}x_{2}$, then using $x_{4}x_{2}=x_{3}x_{4}$, we deduce that $x_{4}=x_{3}^{2}$, which is again not possible. \end{enumerate} 
   Therefore $x_{3}x_{4}\notin X$. Thus $$S^{2}=X\cup \{x_{3}x_{4}\}.$$ Since $x_{4}\in S^{2}$, we have $x_{4}\in\{x_{2}x_{3}, x_{3}x_{2}\}$. If $x_{4}=x_{3}x_{2}=\max(A^{2})$, then $x_{4}^{2}\notin A^{2}\cup \{x_{3}x_{4},x_{2}x_{4}, x_{4}x_{2}\}$. This give $|S^{2}|\geq 10$, which is a contradiction. Therefore $x_{4}=x_{2}x_{3}$. This implies that  $x_{4}^{2}\in \{x_{3}^{2}, x_{3}x_{2}\}$. 
        \begin{enumerate}
            \item If  $x_{4}^{2}=x_{3}^{2}$, then $x_{2}x_{3}x_{4}=x_{4}^{2}=x_{3}^{2}=x_{2}^{2}$. This implies  $x_{2}=x_{3}x_{4}>x_{4}$, which is not possible.
            \item If  $x_{4}^{2}=x_{3}x_{2}$ then $$x_{2}x_{4}=x_{2}^{2}x_{3}=x_{3}^{3}=x_{3}x_{2}^{2}=x_{4}^{2}x_{2}>x_{4}x_{2}>x_{2}x_{4},$$ which again not possible. 
        \end{enumerate} Hence we again obtain $|S^{2}|\geq 10$, so this case cannot occur.\par
  \textbf{Case 3} $(x_{2}x_{4}<x_{4}x_{2} \ \text{and} \ x_{2}x_{3}>x_{2}x_{4})$.  In this case, we have $$ e<x_{2}<x_{3}<x_{4}<x_{2}x_{4}<x_{2}x_{3}<x_{3}^{2}=x_{2}^{2}<x_{3}x_{2}<x_{4}x_{2}.$$ Thus $$S^{2}=\{e, x_{2},x_{3}, x_{4}, x_{2}x_{4}, x_{2}x_{3}, x_{3}^{2}, x_{3}x_{2}, x_{4}x_{2}\}.$$ Since $x_{4}x_{3}\in S^{2}$, we have  $x_{4}x_{3}=x_{3}x_{2}$. This implies $x_{4}=x_{3}x_{2}x_{3}^{-1}=x_{3}g=x_{2}g^{2}$ and $x_{4}^{2}=x_{3}x_{2}x_{3}^{-1}x_{3}x_{2}x_{3}^{-1}=x_{3}x_{2}^{2}x_{3}^{-1}=x_{3}^{2}$.  
  This gives $S=\{e,x,xg,xg^{2}\}$ with $x=x_{2}$ and $g=x_{2}x_{3}^{-1}$. Therefore, (iii) holds. 
\end{proof}
\begin{prop}\label{S2-geq-13}
    Let $q<-1$ be an integer.
     Let $S=\{x_{1},x_{2},x_{3},x_{4},x_{5}\}$ be a nonempty finite subset of $\mathrm{BS}(1,q)$ with $e=x_{1}<x_{2}<x_{3}<x_{4}<x_{5}$ and  $x_{3}x_{4}=x_{4}x_{3}$. If $A=\{x_{1},x_{2},x_{3},x_{4}\}$ is nonabelian  and  $|A^{2}|=9$, then $|S^{2}|\ge 13$.
\end{prop}
\begin{proof}
 Since $|A^{2}|=9$ and $x_{3}x_{4}=x_{4}x_{3}$, by Proposition \ref{lemma 2-theorem-2.13},  $A$ is of the form  $\{e,x_{2},x_{3},x_{2}^{2}\}$, with $e<x_{2}<x_{3}<x_{2}^{2}$ and   $x_{4}=x_{2}^2=x_{3}^2$. By right ordering in $G$, we have \begin{equation}\label{eq-t1}
         e<x_{2}<x_{3}<x_{2}x_{3}<x_{3}^{2}=x_{2}^{2}<x_{3}x_{2}<x_{4}x_{2}=x_{2}x_{4}<x_{3}x_{4}<x_{4}^{2}<x_{5}x_{4}.
     \end{equation} Therefore, $$A^{2}=\{e, x_{2}, x_{3}, x_{2}x_{3}, x_{3}^{2}, x_{3}x_{2}, x_{4}x_{2}, x_{3}x_{4}, x_{4}^{2}\}.$$ Also,
       $$x_{5}x_{4}\notin A^{2}.$$
   
\textbf{Case 1} $(x_{5}x_{3}\in A^{2}).$ Since $x_{3}x_{4}=x_{4}x_{3}<x_{5}x_{3}$, it follows that  $x_{5}x_{3}=x_{4}^{2}$. Using $x_{4}=x_{2}^{2}=x_{3}^{2}$, we obtain  \begin{equation}\label{Lemma-5.6-Eq-1}
x_{5}=x_{3}^3=x_{4}x_{3}=x_{3}x_{4}, x_{3}x_{5} = x_{5}x_{3}, x_{4}x_{5}=x_{5}x_{4},\ \text{and} \ x_{2}x_{5}\neq x_{5}x_{4}.
\end{equation}  Also, $x_{4}^{2}<x_{5}x_{4}<x_{5}^{2}$, which implies $x_{5}^{2}\notin A^{2}$. Moreover, if $x_{2}x_{5}= x_{5}x_{2}$, then \eqref{Lemma-5.6-Eq-1} implies that $x_{2}x_{3}= x_{3}x_{2}$, which contradicts \eqref{eq-t1}. Therefore, $x_{2}x_{5}\neq x_{5}x_{2}$. A similar argument also implies that $x_{5}x_{2} \neq x_{4}^{2}$, $x_{2}x_{5} \neq x_{4}^{2}$ and $x_{5}x_{2} \neq  x_{3}x_{4}$. Further, $x_{4}x_{2}<x_{5}x_{2}$ and $x_{5}=x_{3}x_{4}<x_{2}x_{5}$ imply that  $\{x_{2}x_{5},x_{5}x_{2}\}\cap  A^{2}=\emptyset$. Thus,   $$|S^{2}|\geq |A^{2}|+|\{x_{2}x_{5},x_{5}x_{2},x_{5}x_{4}, x_{5}^{2}\}|=13.$$\par
\textbf{Case 2} ($ x_{5}x_{3}\notin A^{2} \text{ and }  x_{5}x_{2}\in  A^{2}$). Since $x_{4}x_{2}<x_{5}x_{2}$ and $x_{5}x_{2}\in  A^{2}$, we have $x_{5}x_{2}\in \{x_{3}x_{4}, x_{4}^{2}\}$. Consider the following cases. \begin{enumerate}
    \item  If $x_{5}x_{2}=x_{4}^{2}$, then $x_{5} \in \langle x_{2} , x_{4} \rangle$ and $x_{5} = x_{4}x_{2}$. 
    Therefore, $x_{4}x_{5}=x_{5}x_{4}$ and $x_{2}x_{5}=x_{5}x_{2}$. Further, if  $x_{3}x_{5}= x_{5}x_{3}$, then $$x_{4}x_{2}x_{3} = x_{5}x_{3} = x_{3}x_{5} = x_{3}x_{4}x_{2} = x_{4}x_{3}x_{2},$$ this gives $x_{3}x_{2} = x_{2}x_{3}$, which contradicts \eqref{eq-t1}. Therefore,  $x_{3}x_{5}\neq x_{5}x_{3}$. Since $x_{4}^{2}<x_{4}x_{5}<x_{5}^{2}$, we have $x_{5}^{2}\notin A^{2}$. Also,  $x_{3}x_{2}<x_{4}x_{2}<x_{3}x_{4}x_{2}=x_{3}x_{5}$. This implies that if $x_{3}x_{5}\in A^{2}$, then $x_{3}x_{5}=x_{4}^{2}$. This gives $x_{3}x_{2}=x_{4}$ which is not possible because $x_{4}=x_{2}^{2}<x_{3}x_{2}$. Thus $x_{3}x_{5}\notin A^{2}$. Therefore, we have $$|S^{2}|\geq |A^{2}|+|\{x_{3}x_{5},x_{5}x_{3},x_{5}x_{4},x_{5}^{2}\}|=13.$$
\item If $x_{5}x_{2}=x_{3}x_{4}$, then $x_{5}x_{2}=x_{3}x_{2}^{2}$ and this implies  $x_{5}=x_{3}x_{2}$. It follows that $$x_{4}x_{5}=x_{4}x_{3}x_{2}=x_{3}^{2}x_{3}x_{2}=x_{3}x_{2}^{2}x_{2}=x_{5}x_{4}.$$ Now, if $x_{2}x_{5} = x_{5}x_{2}$, then $$x_{3}x_{2}^{2} = x_{3}x_{4} = x_{5}x_{2} = x_{2}x_{5} =  x_{2}x_{3}x_{2},$$
this gives $x_{3}x_{2}=x_{2}x_{3}$, which contradicts \eqref{eq-t1}. Therefore, $x_{2}x_{5}\neq x_{5}x_{2}$. Since $x_{4}^{2}<x_{5}x_{4}=x_{4}x_{5}<x_{5}^{2}$, we have   $x_{5}^{2}\notin A^{2}$. Also, note that $x_{3}x_{2}<x_{2}x_{3}x_{2} = x_{2}x_{5}$ and $x_{3}x_{4} \neq x_{2}x_{5}$. Therefore, if  $x_{2}x_{5}\in A^{2}$, then   $x_{2}x_{5}=x_{4}^{2}$. This implies $x_{2}x_{3}x_{2}=x_{4}x_{2}^{2}$, which gives $x_{3}=x_{4}$, which is not possible. Thus, $x_{2}x_{5}\notin A^{2}$. Now, if $x_{2}x_{5}\neq x_{5}x_{3}$, then $$|S^{2}|\geq  |A^{2}|+|\{x_{2}x_{5},x_{5}x_{3},x_{4}x_{5},x_{5}^{2}\}| =13.$$ 
 Assume that $x_{2}x_{5}=x_{5}x_{3}$. Then $x_{3}x_{5} \neq x_{5}x_{3}$ and $x_{4}x_{3}<x_{5}x_{3}=x_{2}x_{5}<x_{3}x_{5}$. Therefore, if $x_{3}x_{5}\in A^{2}$, then $x_{3}x_{5}=x_{4}^{2}$, which gives $x_{5}=x_{3}^3$, which contradicts $x_{3}x_{5} \neq x_{5}x_{3}$. Thus, $x_{3}x_{5}\notin A^{2}$.  Hence,
 $$|S^{2}|\geq |A^{2}|+|\{x_{3}x_{5},x_{5}x_{3},x_{4}x_{5},x_{5}^{2}\}|=13.$$\end{enumerate}
 
  \textbf{Case 3} ($x_{5}x_{3}\notin A^{2} \text{ and } x_{5}x_{2} \notin  A^{2}$). If $x_{5}\notin A^{2}$ , then $$|S^{2}|\geq |A^{2}|+|\{x_{5}, x_{5}x_{2},x_{5}x_{3},x_{5}x_{4}\}|=13.$$ If $x_{5}\in A^{2}$, then  $x_{5}\in\{x_{3}x_{2},x_{4}x_{2},x_{3}x_{4},x_{4}^{2}\}$. By a similar argument as given in previous cases, we obtain that $x_{4}x_{5}=x_{5}x_{4}$. This gives $x_{5}^{2}\notin A^{2}$. Hence, $$|S^{2}|\geq |A^{2}|+|\{x_{5}^{2},x_{5}x_{2},x_{5}x_{3},x_{5}x_{4}\}|=13.$$ 
\end{proof}

\begin{prop}\label{main lemma} Let $k\geq 5$ be an integer. 
Let $A=\{e,x,xg\ldots,xg^{k-3}\}$ be a subset of a right-ordered group $G$, where $e<x<xg<\cdots<xg^{k-3}$, $x^{2}=(xg^{i})^{2}$ for all $i\in[\![0,k-3]\!]$, and $|A^2|=3|A|-3$.  Let $x_{k}\in G$ such that $x_{k}>xg^{k-3}$ and $|(A\cup\{x_{k}\})^{2}|=3|A|$. Then $x_{k}=xg^{k-2}$ and $x_{k}^{2}=x^{2}$.
\end{prop}
\begin{proof} Let $x_{1}=e, x_{i+2} = xg^{i}$ for $i\in[\![0,k-3]\!]$.  Let $S=A\cup\{x_{k}\}$ with $|S|=k$. Since $x_{3}^{2} = x_{2}^{2}$,  we have  $g=x_{2}^{-1}x_{3}=x_{2}x_{3}^{-1}<e$. Note that, by right-ordering in $G$, we have $$e<x_{2}<\cdots<x_{k-1}<x_{2}x_{k-1}<\cdots<x_{k-1}^{2}=x_{2}^{2}<x_{3}x_{2}<\cdots<x_{k-1}x_{2}.$$ Therefore, $$A^{2}=\{e, x_{2}, \ldots, x_{k-1}, x_{2}x_{k-1}, \ldots, x_{k-1}^{2}, x_{3}x_{2}, \ldots, x_{k-1}x_{2}\}.$$
     Since $x_{k-1}x_{2}<x_{k}x_{2}$, we have  $x_{k}x_{2}\notin A^{2}$. Further, if $x_{2}x_{k}\geq x_{k}x_{2}$, then we have $$\max(A^{2})=x_{k-1}x_{2} < x_{k}x_{2} \leq x_{2}x_{k} < x_{3}x_{k}<x_{4}x_{k}\cdots < x_{k-1}x_{k}<x_{k}^{2},$$ which implies that $|S^{2}|\geq 3k-2$. Therefore, 
     \begin{equation}\label{Theorem-2.11-Eq-1}
         x_{2}x_{k} <x_{k}x_{2}.
     \end{equation}
Now, we show that $x_{2}x_{k}\notin A^{2}$. Suppose that $x_{2}x_{k}\in A^{2}$. Since $x_{k}<x_{2}x_{k}$, there exist $i\in[\![3,k-1]\!]$ and  $j\in[\![2,k-1]\!]$ such that $x_{2}x_{k}=x_{i}x_{j}$.  This gives $x_{2}x_{k}=x_{2}g^{i-2}x_{j}$. Since $g<e$, we have  $x_{k}=g^{i-2}x_{j}<x_{j} \leq x_{k-1}$, which is a contradiction. Therefore,
\begin{equation}\label{Theorem-2.11-Eq-2}
    x_{2}x_{k}\notin A^{2}.
\end{equation}
With a similar argument, we have  \begin{equation}\label{Theorem-2.11-Eq-4}
    x_{3}x_{k}\neq x_{i}x_{j} \ \text{ for} \  i\in [\![3,k-1]\!] \ \text{and} \  j\in [\![2,k-1]\!],
\end{equation}
and
\begin{equation*}\label{Theorem-2.11-Eq-5}
    x_{4}x_{k}\neq x_{i}x_{j} \ \text{for} \  i\in [\![4,k-1]\!] \ \text{and} \  j\in [\![2,k-1]\!].
\end{equation*}
Next, we show that $x_{3}x_{k}\neq x_{k}x_{2}$. Suppose that $x_{k}x_{2}=x_{3}x_{k}$. Then $$\max(A^{2})=x_{k-1}x_{2} <x_{k}x_{2}=x_{3}x_{k}<x_{4}x_{k}<\cdots<x_{k}^{2},$$
     which implies that $|S^{2}|\geq  |A^{2}| + |\{x_{2}x_{k}, x_{k}x_{2},x_{4}x_{k},\ldots, x_{k}^{2}\}|\geq 3k-6+4 =3k-2$. Therefore, 
     \begin{equation}\label{Theorem-2.11-Eq-6}
         x_{3}x_{k}\neq x_{k}x_{2}.
     \end{equation}

   \textbf{Claim} $x_{2}x_{k}<x_{2}x_{k-1}$. 
    
Suppose $x_{2}x_{k-1}<x_{2}x_{k}$. Since $x_{2}x_{k}<x_{3}x_{k}$, it follows from \eqref{Theorem-2.11-Eq-1}, \eqref{Theorem-2.11-Eq-2},  \eqref{Theorem-2.11-Eq-4}, and \eqref{Theorem-2.11-Eq-6}, we have  $$S^2 = A^2 \cup \{x_{3}x_{k},x_{2}x_{k}, x_{k}x_{2}\}.$$
      Note that, if $x_{4}x_{k}=x_{k}x_{2}$, then $\max(A^{2})  = x_{k-1}x_{2}<x_{k}x_{2}<x_{5}x_{k}$. This implies $$|S^{2}| \geq |A^2| \cup |\{x_{3}x_{k},x_{2}x_{k}, x_{k}x_{2},x_{5}x_{k}\}| \geq 3k-2.$$ Therefore, $x_{4}x_{k}\neq x_{k}x_{2}$.  Since $x_{4}x_{k} \in S^{2}$ and $x_{4}x_{k} \notin \{x_{3}x_{k},x_{2}x_{k}, x_{k}x_{2}\}$, we have  $x_{4}x_{k}\in A^{2}$. Similarly,  $x_{k} \in A^{2}$.  Since $x_{2}x_{k-1}<x_{2}x_{k}<x_{3}x_{k}<x_{4}x_{k}$ and $$x_{4}x_{k}\notin \{x_{4}x_{k-1},\ldots, x_{k-1}^{2}, x_{4}x_{2},\ldots, x_{k-1}x_{2}\},$$ we have $x_{4}x_{k}\in \{x_{3}x_{k-1}, x_{3}x_{2}\}$. Suppose that $x_{4}x_{k}=x_{3}x_{2}$. Then we obtain $x_{2}g^{2}x_{k}=x_{2}gx_{2}$, which leads to $x_{k}=x_{3}$ because $g=x_{2}x_{3}^{-1}$, which is a contradiction.
          Therefore, $$x_{4}x_{k}= x_{3}x_{k-1}.$$
     Moreover, $x_{k}<x_{2}x_{k}<x_{4}x_{k}$ and  $x_{k}\in A^{2}$ implies that $x_{k}=x_{2}x_{k-1}$. Thus, we obtain $x_{4}x_{2}x_{k-1}=x_{3}x_{k-1}$, so $x_{3} = x_{4}x_{2}>x_{4}$, which is again a contradiction. Therefore, $$x_{2}x_{k}<x_{2}x_{k-1}.$$ By right-ordering in $G$, we have $$x_{1}<x_{2}<\cdots<x_{k}<x_{2}x_{k}<x_{2}x_{k-1}<\cdots<x_{k-1}^{2}=x_{2}^{2}<x_{3}x_{2}<x_{4}x_{2}<\cdots< x_{k}x_{2}.$$  Therefore $$S^{2}=\{e,x_{2},\ldots, x_{k},x_{2}x_{k}, x_{2}x_{k-1},\ldots, x_{k-1}^{2},x_{3}x_{2},\ldots, x_{k}x_{2}\}.$$
Since  $x_{3}x_{k}\in S^{2}$, $x_{2}x_{k}<x_{3}x_{k}$, $x_{3}x_{k} \neq x_{i}x_{k-1}$, where $i\in [\![3,k-1]\!]$, and $x_{3}x_{k}\neq x_{j}x_{2}$, where $j\in [\![4,k]\!]$, we have $x_{3}x_{k}=x_{2}x_{k-1}$. This implies $x_{k}=x_{2}g^{k-2}$. Therefore,  $$S=\{e, x,xg,\ldots, xg^{k-2}\},$$ with $x=x_{2}$ and $g=x_{2}x_{3}^{-1}$. Note that $$x_{k}^{2} = xg^{k-2}xg^{k-2} = xgx^{-1}(xg^{k-3})^{2}g=xgx^{-1}x^{2}g=(xg)^{2}.$$ This completes the proof.
\end{proof}
\begin{prop}\label{For k=5}
Let $q<-1$ be an integer. Let $S$ be a nonempty finite  subset of a right-ordered group $\mathrm{BS}(1,q)$ with $|S|=5$, $\min(S)=e$, and  $|S^{2}|=12$. Then, either $\langle S\rangle$ is abelian, or there exist $x, g \in \mathrm{BS}(1,q)$ such that 
 $$S=\{e, x, xg, xg^{2}, xg^{3}\},$$ with $e<x<xg<xg^{2}<xg^{3}$ and  $(xg^{2})^{2}=(xg)^{2}=x^{2}$.
    \end{prop}
     \begin{proof}Let $S=\{x_{1},x_{2},x_{3},x_{4},x_{5}\}$ with $e=x_{1}<x_{2}<x_{3}<x_{4}$. Let $A=\{e, x_{2}, x_{3}, x_{4}\}$. If $\langle A \rangle $ is abelian, then it follows from Proposition \ref{prop-2} that either  $\langle S\rangle$ is abelian or  $S=\{e,g, g^{2}, g^{3}, x_{5}\}$,  with $[g^{i},x_{5}]=g^{3-2i}$ for $i\in  [\![ 0,3 ]\!]$. But for $i=0$, we obtain $g=e$, which is not possible. Therefore, $\langle S\rangle$ is abelian if $\langle A \rangle $ is abelian. Now, assume that  $\langle A \rangle$ is nonabelian. Then, Proposition \ref{NM-k=4} implies that $|A| \geq 9$. Further, if  $|A^{2}|\geq 10$, then by Proposition \ref{Sec-5-Lemma-1}, we have  $|S^{2}|\geq 10+3=13$.   Therefore, $|A^{2}|=9$. By Proposition \ref{lemma 2-theorem-2.13}, the set $A$  is either  of the form $\{e,x,xg,(xg)^{2}\}$ or $\{e,x,xg,xg^{2}\}$. If $A=\{e,x,xg,(xg)^{2}\}$, then Proposition \ref{S2-geq-13} implies that  $|S^{2}|\geq 13$. Therefore,   $$A=\{e,x,xg,xg^{2}\},$$ with $e<x<xg<xg^{2}$ and $(xg^{2})^{2}=(xg)^{2}=x^{2}$. Thus, $x_{2}=x$,  $x_{3} = xg$ and $x_{4} = xg^{2}$ with $g=x_{2}^{-1}x_{3}=x_{2}x_{3}^{-1}<e$. By  Proposition \ref{main lemma}, we get $x_{5}=x_{3}g^{2}=x_{2}g^{3}$. Therefore, we have  $$S=\{e,x,xg,xg^{2},xg^{3}\},$$ 
      with $e<x<xg<xg^{2}<xg^{3}$ and  $(xg^{2})^{2}=(xg)^{2}=x^{2}$.
     This completes the proof.
\end{proof}
\begin{proof}[Proof of Theorem \ref{Theorem-3}] If $q \geq 0$, then $BS(1,q)$ is an ordered group (see \cite{Neetu_Mohan_Shankar}), and thus the result follows directly from Theorem \ref{ordered 3k-3}. Now suppose $q < -1$ and let $S = \{x_{1}, x_{2}, \ldots, x_{k-1}, x_{k}\}$ with $e = x_{1} < x_{2} < \cdots < x_{k}$.  We prove by induction on $|S|=k\geq 5$. For $k = 5$, the result follows from Proposition \ref{For k=5}. Assume that the result holds for some $k-1 \geq 5$. Now, we show that the result holds for $|S| = k \geq 6$.

Let $A=\{e, x_{2}, \ldots, x_{k-1}\}.$ If $A$ is abelian, then by Proposition \ref{prop-2} either  $\langle S\rangle$ is  abelian, or  $S=\{e,g,\ldots,g^{k-2}, x_{k}\}$  with $[g^{i},x_{k}]=g^{k-2-2i}$ for $i\in  [\![ 0,k-2 ]\!]$. But for $i = 0$, we get that $[e,x_{k}]=g^{k-2}$, and this implies $g=e$, which is not possible. Therefore, $\langle S\rangle$ is abelian. Now, assume that  $A$ is nonabelian. Therefore, by Proposition \ref{Sec-5-Lemma-1}, we have $$|A^{2}|\leq 3k-3-3=3(k-1)-3.$$ Further, Theorem \ref{Theorem for BS(1,q)} implies that $|A^{2}| \geq 3(k-1)-3.$ Therefore $|A^{2}| = 3(k-1)-3.$ Thus, by induction, we have $$A=\{e,x, xg,\ldots, xg^{k-3}\},$$ with $xg\neq gx$ and    $(xg^{i})^{2}=x^{2}$ for each $i\in [\![0, k-3]\!]$. Therefore $x_{1}=e$ and  $xg^{i}=x_{i+2}$ for $i\in[\![0,k-3]\!]$. Since $x_{3}^{2} = x_{2}^{2}$,  we have  $g=x_{2}^{-1}x_{3}=x_{2}x_{3}^{-1}<e$. By Proposition \ref{main lemma}, we get  $x_{k}=x_{2}g^{k-2}$. Therefore, we have  $$S=\{e, x,xg,\ldots, xg^{k-2}\},$$ with $xg\neq gx$ and $(xg^{i})^{2}=x^{2}$ for each $0\leq i\leq k-2$. This completes the proof.

\end{proof}

\section{Future Problems}\label{section-7}
Theorem \ref{NMS-Thm-3.3} and Theorem  \ref{Theorem-2} motivate us to propose the following conjecture.
\begin{conjecture}
                Let $S$ be a nonempty finite subset of a torsion-free group $G$ such that $S = A \cup B$, where $A$ and $B$ are disjoint abelian sets with $|A|\geq 3$ and $|B|\geq 3$. If $|S^{2}| \leq  3|S|-3$, then $\langle S \rangle$ is abelian.
           \end{conjecture}
\noindent Whereas this may not hold for three abelian sets. Consider the following example.  

           \begin{example} 
               Let
               \begin{align*}
                   A &=\{(-1,1),(-1,3),\ldots,(-1,2k-1)\},\\
                   B&=\{(1,1),(1,3),\ldots,(1,2k-1)\},\\
                   C&=\{(-1,2),(1,2),(-1,4),(1,4),\ldots,(-1,2k),(1,2k)\}.
               \end{align*}
               be subsets of the group $\mathbb{Z}\rtimes \mathbb{Z}$. Let $S=A\cup B\cup C$. Then, we have  
               $$|S^{2}|=3|S|-3.$$ It is easy to observe that  $\langle A\rangle$, $\langle B \rangle$, and $\langle C\rangle$ are abelian, whereas $\langle S\rangle$ is nonabelian.
           \end{example}
      \noindent This motivates us to propose the following conjecture.
           \begin{conjecture}
                Let $S$ be a nonempty finite subset of a torsion-free group $G$ such that $S = A \cup B \cup C$, where $A, B, C$ are pairwise disjoint abelian sets each of cardinality at least three. If $|S^{2}| \leq  3|S|-4$, then $\langle S \rangle$ is abelian.
           \end{conjecture}  
           Also, we strongly believe that the following conjecture is true. 
           \begin{conjecture}
           Let $S$ be a nonempty finite subset of a torsion-free group $G$ with $\left|S\right| \geq 3$ and $e\in S$. If $\left|S^{2}\right| \leq 3|S|-4$, then $\langle S\rangle$ is  abelian.
                
           \end{conjecture}
           The following example shows that $e \in S$ is necessary. 
           \begin{example} Let $S=\{(1,1), (2,1),\ldots,(k,1)\}\subseteq \mathbb{Z}\rtimes \mathbb{Z}$, where $k\geq 3$. Clearly, $S$ is a nonabelian set.  For each $i,j\in [\![1,k]\!]$, we have $$(i,1)^{2} = (0,2), \quad (i,1)(j,1) = (i-j,2), \quad (j,1)(i,1) = (j-i,2).$$
		Thus, $$S^{2} = \{(1-k,2), (2-k,2), \ldots,(k-2,2), (k-1,2)\} ~\text{and}~|S^{2}|=2k-1.$$
	\end{example}
	\bibliographystyle{amsplain}

\end{document}